\documentclass[final,1p]{elsarticle}
\usepackage{fancyhdr,graphicx}
\usepackage{amsmath,amssymb}
\allowdisplaybreaks[3]
\usepackage{latexsym}
\usepackage{color}
\usepackage{subfig}
\usepackage{algorithm}
\usepackage{algorithmic}
\usepackage{epstopdf}
\usepackage{multirow}
\usepackage{longtable}
\usepackage{cases}
\usepackage[colorlinks,linkcolor=blue,hyperindex]{hyperref}

\usepackage[pagewise]{lineno}%\linenumbers
\newtheorem{theorem}{Theorem}
\newtheorem{lemma}{Lemma}
\newtheorem{remark}{Remark}

\newtheorem{definition}{Definition}

\numberwithin{equation}{section}
\usepackage{epsfig}
\usepackage{subfig}
\journal{ }
\begin{document}
\begin{frontmatter}

% \title[short title]{TITLE}
\title{Inverse Source Problems for a Class of Fractional Elliptic Equations with Singular Coefficients}

\author[GMUSCI]{Zewen Wang}
\ead{zwwang@ecut.edu.cn; zwwang6@163.com}

\author[ECUT,GMUSCI]{Liang Guo}
\ead{2714901483@qq.com}

\author[GMUSCI,GMUAI]{Shufang Qiu \texorpdfstring{\corref{cor1}}{}}% 
\ead{qiushufang@gzmtu.edu.cn; shfqiu@163.com}

\author[NUIST]{Bin Wu}

\cortext[cor1]{Corresponding author.}

\address[GMUSCI]{School of Arts and Sciences, Guangzhou Maritime University, Guangzhou 510725, P. R. China}

\address[ECUT]{School of Science, East China University of Technology, Nanchang, Jiangxi, 330013, P. R. China}

\address[GMUAI]{School of Artificial Intelligence, Guangzhou Maritime University, Guangzhou 510725, P. R. China}

\address[NUIST]{School of Mathematics and Statistics, Nanjing University of
Information Science
and Technology, Nanjing 210044, P. R. China}

% ----------------------------------------------------------------
\begin{abstract}
An inverse source problem for a class of fractional elliptic equations with singular coefficients is investigated in this paper. For the corresponding direct problem, a formal solution is derived and the well-posedness of the solution is established. For the inverse problem, a H\"older-type conditional stability estimate is obtained in a Hilbert scale associated with exponential operators. Based on this stability framework, two regularization methods are proposed for reconstructing the unknown source term: the exponential-type Tikhonov regularization method and the exponential quasi-boundary value regularization method. Convergence estimates for the regularized solutions are derived under both a priori and a posteriori choices of the regularization parameter. In addition, finite-dimensional spectral approximation results show that the proposed methods are also applicable to general square-integrable source terms, without requiring the exact source to satisfy an exponential-type source condition. Numerical experiments demonstrate that the proposed methods provide stable and accurate reconstructions for both smooth and piecewise smooth sources even under low signal-to-noise ratio conditions. 
\end{abstract}


\end{frontmatter}

%%%%%%%%%%%%%%%%%%%%%%%%%%%%%%%%%%%%%%%%%%%%
%% MAINMATTER
%%%%%%%%%%%%%%%%%%%%%%%%%%%%%%%%%%%%%%%%%%%%

%==============================================
\section{Introduction} \label{sec:introduction}
%==============================================
Fractional calculus has grown into an important branch of modern analysis. The books by Podlubny \cite{ref01} and by Kilbas, Srivastava, and Trujillo \cite{ref02} give clear and systematic introductions to fractional derivatives, fractional differential equations, the main tools for solving them, and many applications. In particular, \cite{ref02} stresses the importance of fractional elliptic (degenerate/weighted) structures and shows how algebraic weights can change well-posedness and the form of solutions in a fundamental way. A simple but representative model is 
$\partial_{0,x}^{\alpha} y(x) - \lambda x^{\beta} y(x) = 0$, 
whose behavior is controlled by the parameter $\beta$. When $\beta=0$, the model can be extended to a nonhomogeneous equation with a source term $f(x)$, and standard tools such as the Laplace transform and convolution formulas work well. When $\beta\neq 0$, however, the factor $x^{\beta}$ makes these tools harder to use, and adding a source term becomes much more delicate. 

Fractional elliptic and mixed-type equations appear in many scientific and engineering settings. Examples include subsonic and transonic aerodynamics \cite{ref03}, aspects of explosion dynamics and nonlinear waves \cite{ref04}, and microscopic problems in Fermi-liquid theory where Keldysh-type ideas are relevant \cite{ref05}. Within this research field, several classical mixed-type models---the Tricomi \cite{ref06}, Gellerstedt \cite{ref07}, and Keldysh \cite{ref08} equations---have inspired extensive research. On the theory side, results on global existence, exact solutions in special geometries, and nonexistence for related models can be found in \cite{ref09,ref10,ref11,ref12}. More recently, well-posedness for fractional versions of Tricomi–Gellerstedt–Keldysh-type problems was established in \cite{ref13}. Inverse and regularization questions for fractional elliptic models have also received growing attention; see \cite{DjemouiS01,ZhangX01}.

Motivated by the above developments, we study a fractional elliptic equation with singular coefficients:
\begin{equation}\label{equ101}
\frac{1}{x^{\beta}}\,
\partial_{0,x}^{\alpha}\!\left(
\frac{1}{x^{\beta}}\,
\partial_{0,x}^{\alpha} u(x,y)
\right)
+ L u(x,y) = g(y),
\end{equation}
where $\alpha>0$, $\beta\in\mathbb{R}$, $-L$ is a symmetric uniformly elliptic operator acting on the transverse variable $y$, and $g$ depends only on $y$. The weights $x^{-\beta}$ encode a singular/degenerate structure similar to Keldysh-type operators, and hence \eqref{equ101} naturally belongs to the fractional Tricomi-Gellerstedt-Keldysh family \cite{ref06,ref07,ref08,ref13}. It should be emphasized that the weighted fractional structure in \eqref{equ101} is different from the equation considered in \cite{ref13}, where the singular coefficient appears as a multiplicative factor before the elliptic operator. Therefore, the separated fractional ordinary differential equation and the corresponding fundamental solution representation require a separate derivation in the present work.

In particular, when $\alpha=1$ and $L=\partial^2/\partial y^2$, equation \eqref{equ101} becomes
\[
x^{-2\beta}u_{xx}(x,y)-\beta x^{-2\beta-1}u_x(x,y)+u_{yy}(x,y)=g(y),
\]
or equivalently,
\[
u_{xx}(x,y)-\beta x^{-1}u_x(x,y)+x^{2\beta}u_{yy}(x,y)
=
x^{2\beta}g(y),
\]
which can be regarded as a variant of the Gellerstedt-type equation with an additional first-order singular term. We first prove the well-posedness of the direct problem under suitable assumptions on $(\alpha,\beta)$, the domain, and the boundary data. We then turn to the inverse source problem for \eqref{equ101}, aiming to recover the forcing term from interior observations.

The remainder of this paper is organized as follows.  
In Section 2, we formulate the inverse source problem under consideration and introduce the necessary preliminaries, including the required fractional operators and functional settings.  
Section 3 is devoted to proving the well-posedness of the corresponding direct problem.  
In Section 4, we investigate the inverse problem, show that it is ill-posed in the sense of Hadamard, and derive a conditional stability estimate. Section 5 proposes two regularization methods for reconstructing the source term: the exponential-type Tikhonov regularization method and the exponential quasi-boundary value regularization method. Convergence results for the corresponding regularized solutions are established under both a priori and a posteriori parameter choice strategies. Section 6 presents two numerical examples to illustrate the effectiveness of the proposed method: one with a smooth source function and the other with a piecewise smooth source. Finally, Section 6 concludes the paper.

%==============================================
\section{Problem formulation and Preliminaries} \label{sec:Preliminary}
%==============================================
In this section we state the model and data, and gather the main analytical ingredients needed later. We specify the domain and the symmetric uniformly elliptic operator $-L$, introduce the Caputo derivative and the Riemann–Liouville integral, and define the singularly weighted fractional operators that enter our equation. We also recall properties of the three-parameter Mittag–Leffler (Kilbas–Saigo) function \cite{Kilbas1995}, together with exponential-type operators and the Hilbert spaces they generate. These preliminaries set the stage for the precise formulation of the boundary value problem and for the analysis in Sections 3–5.

\subsection{Problem formulation}
Let $\Omega\subset \mathbb{R}^{N}(N\geq1)$ be a bounded connected domain with smooth boundary, denoted by $\partial\Omega$. Let $\alpha\in(0,1]$ be a constant, and $\partial_{0,x}^{\alpha}$ denote the $\alpha$-order Caputo fractional derivative with respect to the variable $x$ \cite{ref01}  defined by
\begin{equation}\label{equa101}
\partial_{0,x}^\alpha u(x,y)=
\begin{cases}
\frac1{\Gamma(1-\alpha)}\int_0^x (x-s)^{-\alpha}\partial_su(s,y)ds, & \alpha\in(0, 1),\\
\partial_x u(x,y), & \alpha = 1
\end{cases}
\end{equation}
where $\partial_x u(x,y)$ denotes the first-order partial derivative of $u(x,y)$ with respect to $x$. 
$I^\alpha_{0,x}$ denotes the $\alpha$-order Riemann-Liouville fractional integral with respect to the variable $x$ \cite{ref01} defined by
\begin{equation}\label{equa102}
I^{\alpha}_{0,x}u(x,y)=\frac{1}{\Gamma(\alpha)}\int_{0}^{x}(x-t)^{\alpha-1}u(t,y)dt. 
\end{equation}
Thus, we have $\partial_{0,x}^\alpha u(x,y)=I^{1-\alpha}_{0,x} \partial_x u(x,y)$. For $x\in(0,+\infty)$ and $\beta\geq 0$, we further define
$$D_{x}^{\alpha}=\frac{1}{x^\beta}\partial_{0,x}^{\alpha},\; D_{x}^{2\alpha}=D_{x}^{\alpha}D_{x}^{\alpha}=\frac{1}{x^\beta}\partial_{0,x}^{\alpha}\frac{1}{x^\beta}\partial_{0,x}^{\alpha}. $$

In this paper, we consider the boundary value problem for a nonhomogeneous fractional elliptic equation with singular coefficients:
\begin{equation}\label{equa103}
\begin{cases}
D_x^{2\alpha}u(x,y)+Lu(x,y)=g(y), & (x,y)\in(0,\infty)\times\Omega,\\
u(x,y)=0, & (x,y)\in(0,\infty)\times\partial\Omega,\\
u(0,y)=f(y), & y\in\Omega,\\
\displaystyle \lim\limits_{x\to\infty}\|u(x,\cdot)\|_{L^2(\Omega)}<\infty. &
\end{cases}
\end{equation}
Here $\alpha\in(0,1]$, $\beta\geq 0$, and
$
-L:\; H^2(\Omega)\cap H_0^1(\Omega)\to L^2(\Omega)
$
is symmetric and uniformly elliptic. Let $(\lambda_k,\varphi_k)$ be the eigenpairs of $-L$:
\[
-L\varphi_k=\lambda_k\varphi_k,\quad k=1,2,\ldots.
\]
By the spectral theory of symmetric uniformly elliptic operators, the eigenvalues satisfy
\[
0<\lambda_1 \leq \lambda_2 \leq \cdots \leq \lambda_k \leq \cdots,
\quad \lim_{k\to\infty}\lambda_k = +\infty,
\]
and the eigenfunctions $\{\varphi_k\}_{k=1}^\infty$ form a complete orthonormal basis of $L^2(\Omega)$.

The inverse problem we study is to reconstruct the unknown source $g(y)$ in \eqref{equa103} from the model and a measurement at $x=l$:
\begin{equation}\label{equa104}
u(l,y)=h^\delta(y), \quad y\in\Omega,
\end{equation}
where $h^\delta$ is a noisy observation of the exact (noise-free) trace $h(y)=u(l,y)$. The noise level $\delta>0$ is assumed known and satisfies
\[
\|h^\delta - h\|_{L^2(\Omega)} \le \delta.
\]
Our analysis first establishes well-posedness of the direct problem and then provides a stable framework for recovering $g$ from the data $h^\delta$.

Since the solution representation of the direct problem \eqref{equa103} involves the three-parameter Mittag-Leffler function, also known as the Kilbas-Saigo function, the recovery of the source term $g(y)$ from the noisy trace data $h^\delta(y)$ is highly unstable. To overcome this ill-posedness, we first develop an exponential-type Tikhonov regularization method inspired by \cite{ref16,Yu2023}, where a similar exponential-type strategy was used to stabilize inverse source problems for fractional diffusion equations. Furthermore, motivated by the modified quasi-boundary value regularization method in \cite{Hao2019,Wei2022,ref20}, we propose an exponential quasi-boundary value regularization method for the present fractional elliptic model with singular coefficients. To the best of our knowledge, the proposed exponential quasi-boundary value regularization method has not been seen in the literature. For these two methods, we derive convergence estimates under both a priori and a posteriori parameter choice rules. In particular, the Morozov discrepancy principle is employed as the a posteriori strategy, which is more suitable for practical applications where the regularity information of the exact source term is usually unavailable.	

\subsection{Preliminaries}
\begin{definition}\label{def1}
The three-parameter Mittag-Leffler function $E_{\alpha, m, n}(z)$ is defined as
\begin{equation}\label{equa201}
E_{\alpha,m,n}(z)=1+\sum_{k=1}^{\infty}\prod_{j=0}^{k-1}\frac{\Gamma(\alpha(jm+n)+1)}{\Gamma(\alpha(jm+n+1)+1)}z^{k},\; z\in\mathbb{C},
\end{equation}
where $\alpha, m>0$, $n>-\frac1\alpha.$
\end{definition}

\begin{lemma}\label{lemma1}
For $\alpha\in(0,1]$, $m>0$, $m=1+n$, $z=\mu x^{\alpha m}$, we have
\begin{equation}\label{equa202}
\partial_{0,x}^{\alpha} E_{\alpha,m,n}(z)={\mu}{x^{\alpha m-\alpha}} E_{\alpha,m,n}(z).
\end{equation}
\end{lemma}

\textbf{Proof}. By direct calculation, we obtain
\begin{align*}&\partial_{0,x}^{\alpha}E_{\alpha,m,n}(\mu x^{\alpha m})=\partial_{0,x}^{\alpha}\left[1+\sum_{k=1}^{\infty}\prod_{j=0}^{k-1}\frac{\Gamma(\alpha(jm+n)+1)}{\Gamma(\alpha(jm+n+1)+1)}{\mu}^{k}x^{(\alpha m)k}\right]\\	&=I_{0,x}^{1-\alpha}\frac{d}{dx}\left[1+\sum_{k=1}^{\infty}\prod_{j=0}^{k-1}\frac{\Gamma(\alpha(jm+n)+1)}{\Gamma(\alpha(jm+n+1)+1)}{\mu}^{k}x^{\alpha mk}\right]\\		&=\sum_{k=1}^{\infty}\prod_{j=0}^{k-1}\frac{\Gamma(\alpha(jm+n)+1)}{\Gamma(\alpha(jm+n+1)+1)}{\mu}^{k}\frac{\alpha m k}{\Gamma(1-\alpha)}\int_{0}^{x}(x-t)^{-\alpha}t^{\alpha m k-1}dt\\	&=\sum_{k=1}^{\infty}\prod_{j=0}^{k-1}\frac{\Gamma(\alpha(jm+n)+1)}{\Gamma(\alpha(jm+n+1)+1)}{\mu}^{k}\frac{\alpha m k}{\Gamma(1-\alpha)}\frac{\Gamma(1-\alpha)\Gamma(\alpha m k)}{\Gamma(\alpha m k+1-\alpha)}x^{(\alpha m)k-\alpha}\\	
&={\mu}{x^{\alpha m-\alpha}}\sum_{k=1}^{\infty}\prod_{j=0}^{k-1}\frac{\Gamma(\alpha(jm+n)+1)}{\Gamma(\alpha(jm+n+1)+1)}\frac{\Gamma(\alpha m k+1)}{\Gamma(\alpha m k+1-\alpha)}{\mu}^{k-1} x^{(\alpha m)(k-1)}\\
& = {\mu}{x^{\alpha m-\alpha}}\left[1+\sum_{k=1}^{\infty}\prod_{j=0}^{k-1}\frac{\Gamma(\alpha(jm+n)+1)}{\Gamma(\alpha(jm+n+1)+1)}{\mu}^{k}x^{(\alpha m)k}\right]. 
\end{align*}
The last equality follows from the fact that, for $m = 1 + n$, we have
\[
\alpha\bigl((k-1)m + n + 1\bigr) + 1 = \alpha mk + 1,
\quad \text{and} \quad
\alpha\bigl((k-1)m + n\bigr) + 1 = \alpha mk - \alpha + 1.
\]
Hence, we obtain
\[
\partial_{0,x}^{\alpha} E_{\alpha,m,n}(\mu x^{\alpha m})
= \mu\, x^{\alpha m - \alpha} \, E_{\alpha,m,n}(\mu x^{\alpha m}),
\]
which completes the proof.
\hfill $\Box$

\begin{lemma}\label{lemma2}
Let $\alpha \in (0,1]$, $m>0$, and $q = n+1-m > 0$. Then the following identity holds:
\begin{equation}\label{equa203}
\partial_{0,x}^{\alpha} \left(x^{\alpha q} E_{\alpha,m,n}(\mu x^{\alpha m}) \right)
=
\mu\, x^{\alpha n} E_{\alpha,m,n}(\mu x^{\alpha m})
+
\frac{\Gamma(\alpha q + 1)}{\Gamma(\alpha q + 1 - \alpha)}\, x^{\alpha (n - m)}.
\end{equation}
\end{lemma}

\textbf{Proof}. 
Using the series representation \eqref{equa201}, we expand
\[
x^{\alpha q} E_{\alpha,m,n}(\mu x^{\alpha m})
= x^{\alpha q}
+ \sum_{k=1}^{\infty}
\left(
\prod_{j=0}^{k-1}
\frac{\Gamma(\alpha(jm+n)+1)}{\Gamma(\alpha(jm+n+1)+1)}
\right)
\mu^k x^{\alpha mk + \alpha q}.
\]
Applying the Caputo derivative $\partial_{0,x}^{\alpha} = I_{0,x}^{1-\alpha} \frac{d}{dx}$ term by term, we obtain
\begin{align*}
\partial_{0,x}^{\alpha}\!&\left(x^{\alpha q} E_{\alpha,m,n}(\mu x^{\alpha m})\right)
= \frac{\Gamma(\alpha q + 1)}{\Gamma(\alpha q + 1 - \alpha)}\, x^{\alpha q - \alpha}  \\
&\quad
+ \sum_{k=1}^{\infty}
\left(
\prod_{j=0}^{k-1}
\frac{\Gamma(\alpha(jm+n)+1)}{\Gamma(\alpha(jm+n+1)+1)}
\right)
\frac{\Gamma(\alpha(mk+q)+1)}{\Gamma(\alpha(mk+q)+1-\alpha)} \mu^k x^{\alpha(mk+q)-\alpha}.
\end{align*}
Using $q = n + 1 - m$, we have:
\[
\alpha(mk + q) + 1 = \alpha m(k-1) + \alpha (n+1) + 1,
\quad
\alpha(mk + q) + 1 - \alpha = \alpha mk + \alpha(n - m) + 1.
\]
Factor out $\mu x^{\alpha n}$, shift the summation index, and recognize again the series of $E_{\alpha,m,n}$ to obtain
\[
\partial_{0,x}^{\alpha} \left(x^{\alpha q} E_{\alpha,m,n}(\mu x^{\alpha m}) \right)
= 
\mu\, x^{\alpha n} E_{\alpha,m,n}(\mu x^{\alpha m})
+
\frac{\Gamma(\alpha q + 1)}{\Gamma(\alpha q + 1 - \alpha)}\, x^{\alpha(n - m)}.
\]
This completes the proof.\hfill $\Box$.

In particular, setting $n=m$ (so $q=1$) in Lemma \ref{lemma2}, we have $x^{\alpha(n-m)}=1$ and
\[
\frac{\Gamma(\alpha q+1)}{\Gamma(\alpha q+1-\alpha)}
=\frac{\Gamma(\alpha+1)}{\Gamma(1)}
=\Gamma(\alpha+1).
\]
Hence,
\[
\partial_{0,x}^{\alpha}\!\bigl(x^{\alpha}E_{\alpha,m,m}(\mu x^{\alpha m})\bigr)
=\mu\,x^{\alpha m}\,E_{\alpha,m,m}(\mu x^{\alpha m})
+\Gamma(\alpha+1),
\]
which is exactly the asserted identity when $n=m$.

\begin{lemma}\label{lemma5}
Let $\alpha\in(0,1]$, $m>0$, $n\geq m-1$, $z\in \mathbb{R}$ and $z\geq 0$. For the three-parameter Mittag-Leffler function $E_{\alpha,m,n}(z)$, the following inequalities hold \cite{ref17}:
\begin{equation}\label{equa205}
\frac{1}{1+\frac{\Gamma(1+\alpha(n-m))}{\Gamma(1+\alpha(n-m+1))}z}
\;\leq\; E_{\alpha,m,n}(-z)\;\leq\;
\frac{1}{1+\frac{\Gamma(1+\alpha n)}{\Gamma(1+\alpha(n+1))}z}, 
\end{equation}
\begin{equation}\label{equ2002}
E_{\alpha,m,n}(z)\;\geq\;\frac{\Gamma(\alpha n+1)}{\Gamma(\alpha(n+1)+1)}\,z .
\end{equation}		
\end{lemma}

From inequality \eqref{equa205}, let us set $z = \mu\,l^{\alpha+\beta}$ with $\mu > 0$ and $l > 0$.  
Then, by choosing 
$
m = 1 + \frac{\beta}{\alpha}, 
\;
n = \frac{\beta}{\alpha},
$ 
we obtain
\begin{equation}\label{equ203}
\frac{1}{1 + \eta_1\,\mu\,l^{\alpha+\beta}}
\;\leq\;
E_{\alpha,\,1+\frac{\beta}{\alpha},\,\frac{\beta}{\alpha}}
\!\left(-\mu\,l^{\alpha+\beta}\right)
\;\leq\;
\frac{1}{1 + \eta_2\,\mu\,l^{\alpha+\beta}},
\end{equation}
where 
\[
\eta_1 = \Gamma(1 - \alpha),
\qquad 
\eta_2 = \frac{\Gamma(1 + \beta)}{\Gamma(1 + \alpha + \beta)}.
\]
Similarly, for 
$
m = 1 + \frac{\beta}{\alpha},
\;
n = 1 + \frac{2\beta}{\alpha},
$ 
we obtain the estimate
\begin{equation}\label{equ208}
\frac{1}{1 + \eta_2\,\mu\,l^{\alpha+\beta}}
\;\leq\;
E_{\alpha,\,1+\frac{\beta}{\alpha},\,1+\frac{2\beta}{\alpha}}
\!\left(-\mu\,l^{\alpha+\beta}\right)
\;\leq\;
\frac{1}{1 + \eta_3\,\mu\,l^{\alpha+\beta}},
\end{equation}
where 
\[
\eta_3 = \frac{\Gamma(1 + \alpha + 2\beta)}{\Gamma(1 + 2\alpha + 2\beta)}.
\]

Next, we introduce several normed function spaces that will be used throughout this paper. 

Let $C(0,+\infty)$ denote the space of all continuous functions on $(0,+\infty)$.  
For a parameter $0 \leq \gamma < 1$, we define
\[
C_{\gamma}(0,+\infty)
=
\left\{
f(x)\,\bigg|\,
x^{\gamma}f(x) \in C(0,+\infty)
\right\},
\]
that is, functions whose weighted form $x^\gamma f(x)$ remains continuous on $(0,+\infty)$. 
Similarly, for $\gamma \le \alpha$, we define the space
\[
C_{\gamma}^{\alpha}(0,+\infty)
=
\left\{
f(x) \in C(0,+\infty)\,\bigg|\,
\partial_{0,x}^{\alpha} f(x) \in C_{\gamma}(0,+\infty)
\right\}.
\]
These spaces are equipped with the norms
\[
\|f\|_{C_\gamma}
= \sup_{x \in (0,+\infty)} \left| x^\gamma f(x) \right|,
\qquad
\|f\|_{C_\gamma^{\alpha}}
= \sup_{x \in (0,+\infty)} |f(x)|
+ \sup_{x \in (0,+\infty)} \left| x^\gamma \partial_{0,x}^{\alpha} f(x) \right|.
\]

For the symmetric uniformly elliptic operator $-L$, we define its exponential operator by the series
\begin{equation}
\exp\bigl((-L)^r\bigr)
= I + \frac{1}{1!}(-L)^r
+ \frac{1}{2!}(-L)^{2r}
+ \frac{1}{3!}(-L)^{3r}
+ \cdots
= \sum_{m=0}^{\infty} \frac{(-L)^{mr}}{m!},
\end{equation}
where $r \in \mathbb{R}$, $I$ is the identity operator, and $(-L)^{mr}$ is defined spectrally via
\[
(-L)^{mr}\varphi_k = \lambda_k^{mr}\varphi_k, \quad m=0,1,2,\dots,
\]
with $\{(\lambda_k,\varphi_k)\}_{k=1}^{\infty}$ being the eigenpairs of $-L$. 
Accordingly, we introduce the following function spaces:
\[
D((-L)^r)
=
\left\{
\psi \in L^2(\Omega)
\,\Big|\,
\sum_{k=1}^{\infty} \lambda_k^{2r} \big|(\psi,\varphi_k)\big|^2 < \infty
\right\},
\]
\[
D\!\left(\exp\!\left(\frac{(-L)^r}{2}\right)\right)
=
\left\{
\psi \in L^2(\Omega)
\,\Big|\,
\sum_{k=1}^{\infty} \exp(\lambda_k^{r}) \big|(\psi,\varphi_k)\big|^2 < \infty
\right\},
\]
where $(\cdot,\cdot)$ denotes the inner product in $L^2(\Omega)$. 

Clearly, both $D((-L)^r)$ and $D\!\left(\exp\!\left(\frac{(-L)^r}{2}\right)\right)$ are Hilbert spaces equipped respectively with the norms
\[
\|\psi\|_{D((-L)^r)}
=
\left( \sum_{k=1}^{\infty} \lambda_k^{2r} \big|(\psi,\varphi_k)\big|^2 \right)^{1/2},
\quad
\|\psi\|_{r,\mathrm{exp}}
=
\left( \sum_{k=1}^{\infty} \exp(\lambda_k^r) \big|(\psi,\varphi_k)\big|^2 \right)^{1/2}.
\]
Particularly, we write $D((-L)^1)$ and $D\!\left(\exp\!\left(\frac{(-L)^1}{2}\right)\right)$ simply as $D(-L)$ and $D\!\left(\exp\!\left(-\frac{L}{2}\right)\right)$, respectively.  
When $r=0$, both spaces reduce to the Hilbert space $L^2(\Omega)$ with norm $\|\cdot\|_{L^2(\Omega)}$.

%==============================================
\section{Well-posedness of the Direct Problem} \label{sec:Well-posedness of the direct problem}
%==============================================
The direct problem associated with \eqref{equa103} consists in determining the solution $u(x,y)$ once the initial data $f(y)$ and the source term $g(y)$ are given.

\begin{theorem}\label{th03}
Assume that $f \in D(-L)$ and $g \in L^2(\Omega)$, and that the compatibility conditions at $x=0$ are satisfied. Then the boundary value problem \eqref{equa103} admits a unique generalized solution given by
\begin{align}\label{equ301}
u(x,y)
&=
\sum_{k=1}^{\infty}
f_k\,
E_{\alpha,\,1+\frac{\beta}{\alpha},\,\frac{\beta}{\alpha}}
\left(-\sqrt{\lambda_k}\,x^{\alpha+\beta}\right)
\,\varphi_k(y) \nonumber\\
&\quad
-\sum_{k=1}^{\infty}
\frac{\Gamma(1+\beta)}{\Gamma(1+\alpha+\beta)}\,
\frac{g_k}{\sqrt{\lambda_k}}\,
x^{\alpha+\beta}
E_{\alpha,\,1+\frac{\beta}{\alpha},\,1+\frac{2\beta}{\alpha}}
\left(-\sqrt{\lambda_k}\,x^{\alpha+\beta}\right)\,
\varphi_k(y),
\end{align}
where
$
f_k = (f,\varphi_k), 
\quad 
g_k = (g,\varphi_k),
\quad k = 1,2,\dots.
$
\end{theorem}

\textbf{Proof.}
Since the eigenfunctions $\{\varphi_k\}_{k=1}^\infty$ of the operator $-L$ form a complete orthonormal basis of $L^2(\Omega)$, the solution $u(x,y)$, the source term $g(y)$, and the boundary value data $f(y)$ can be expanded in the form
\begin{equation}\label{equa302}
\begin{cases}
u(x,y) = \displaystyle\sum_{k=1}^{\infty} u_k(x)\, \varphi_k(y), \\
g(y) = \displaystyle\sum_{k=1}^{\infty} g_k\, \varphi_k(y), \quad
f(y) = \displaystyle\sum_{k=1}^{\infty} f_k\, \varphi_k(y),
\end{cases}
\end{equation}
where the Fourier coefficients are defined by
\[
u_k(x) = (u(x,\cdot), \varphi_k), 
\qquad
g_k = (g, \varphi_k), 
\qquad
f_k = (f, \varphi_k).
\]

Substituting the expansions \eqref{equa302} into \eqref{equa103} and using the orthogonality of $\{\varphi_k\}$, we obtain a sequence of one-dimensional boundary value problems for $u_k(x)$:
\begin{equation}\label{equ303}
\begin{cases}
D_x^{2\alpha} u_k(x) - \mu_k^2 u_k(x) = g_k, \quad x > 0, \\[1mm]
u_k(0) = f_k, \\[1mm]
\displaystyle \lim_{x \to +\infty} |u_k(x)| < \infty,
\end{cases}
\end{equation}
where $\mu_k = \sqrt{\lambda_k}$. 

We first study the homogeneous part of \eqref{equ303}, namely
\begin{equation}\label{equ304}
D_x^{2\alpha} u_k(x) - \mu_k^2 u_k(x) = 0.
\end{equation}
Using the definition $D_x^{2\alpha} = D_x^{\alpha}D_x^{\alpha}$, equation \eqref{equ304} can be rewritten as
\begin{equation}\label{equ305}
\bigl(D_x^{\alpha} - \mu_k\bigr)\bigl(D_x^{\alpha} + \mu_k\bigr) u_k(x) = 0.
\end{equation}
Since the operators $D_x^{\alpha} - \mu_k$ and $D_x^{\alpha} + \mu_k$ commute, equation \eqref{equ305} implies that the general solution of \eqref{equ304} is determined by the solutions of the following two first-order fractional differential equations:
\begin{equation}\label{equ306}
\bigl(D_x^{\alpha} + \mu_k\bigr) u_k(x) = 0,
\end{equation}
or
\begin{equation}\label{equ307}
\bigl(D_x^{\alpha} - \mu_k\bigr) u_k(x) = 0.
\end{equation}

From Theorem 3.25 and formula (4.1.82) in \cite{ref02}, the general solution of \eqref{equ306} is given by  
\[
C_1\,E_{\alpha,\,1+\frac{\beta}{\alpha},\,\frac{\beta}{\alpha}}\!\left(-\mu_k\,x^{\alpha+\beta}\right),
\]
while the general solution of \eqref{equ307} takes the form  
\[
C_2\,E_{\alpha,\,1+\frac{\beta}{\alpha},\,\frac{\beta}{\alpha}}\!\left(\mu_k\,x^{\alpha+\beta}\right),
\]
where $C_1$ and $C_2$ are arbitrary constants.  
Therefore, the general solution of the homogeneous equation \eqref{equ304} can be written as
\begin{equation}\label{equ308}
u_k^0(x)
=
C_1\,E_{\alpha,\,1+\frac{\beta}{\alpha},\,\frac{\beta}{\alpha}}\!\left(-\mu_k\,x^{\alpha+\beta}\right)
+
C_2\,E_{\alpha,\,1+\frac{\beta}{\alpha},\,\frac{\beta}{\alpha}}\!\left(\mu_k\,x^{\alpha+\beta}\right).
\end{equation}

We now show that \eqref{equ308} indeed represents the full set of solutions to \eqref{equ305}.  
Suppose, on the contrary, that there exists a solution $u_k(x)$ of \eqref{equ304} such that
\[
(D_x^{\alpha} + \mu_k)u_k(x) \neq 0, 
\qquad
(D_x^{\alpha} - \mu_k)u_k(x) \neq 0,
\]
but
\[
(D_x^{\alpha} - \mu_k)(D_x^{\alpha} + \mu_k)u_k(x) = 0
\quad\text{and}\quad
(D_x^{\alpha} + \mu_k)(D_x^{\alpha} - \mu_k)u_k(x) = 0.
\]
This implies that both compositions of the operators annihilate $u_k(x)$, while neither factor does. Then, we obtain
$$(D_x^{\alpha} + \mu_k ) u_k(x)=C_2 E_{\alpha,1+\frac{\beta}{\alpha},\frac{\beta}{\alpha}}(\mu_k x^{\alpha+\beta}),$$
$$(D_x^{\alpha} - \mu_k ) u_k(x)=C_1 E_{\alpha,1+\frac{\beta}{\alpha},\frac{\beta}{\alpha}}(-\mu_k x^{\alpha+\beta}).$$
Subtracting the two equations above gives
$$u_k(x) = \frac{1}{2\mu_k}\left[-C_1 E_{\alpha,1+\frac{\beta}{\alpha},\frac{\beta}{\alpha}}(-\mu_k x^{\alpha+\beta})+
C_2 E_{\alpha,1+\frac{\beta}{\alpha},\frac{\beta}{\alpha}}(\mu_k x^{\alpha+\beta})
\right].$$
Since $C_1$ and $C_2$ are arbitrary constants, any solution of the homogeneous equation \eqref{equ305} must be of the form $u_k^0(x)$ given in \eqref{equ308}. 
Furthermore, the functions 
\[
E_{\alpha,\,1+\frac{\beta}{\alpha},\,\frac{\beta}{\alpha}}(-\mu_k x^{\alpha+\beta})
\quad \text{and} \quad
E_{\alpha,\,1+\frac{\beta}{\alpha},\,\frac{\beta}{\alpha}}(\mu_k x^{\alpha+\beta})
\]
are linearly independent. 
Therefore, $u_k^0(x)$ represents the complete set of solutions to the homogeneous equation \eqref{equ305}, and hence also to \eqref{equ304}.

On the other hand, by choosing 
$
m = 1 + \frac{\beta}{\alpha}, 
\; 
n = 1 + \frac{2\beta}{\alpha}
$
in Lemma \ref{lemma2}, we obtain
\[
\partial_{0,x}^{\alpha}\!\left(x^{\alpha+\beta}
E_{\alpha,\,1+\frac{\beta}{\alpha},\,1+\frac{2\beta}{\alpha}}(\mu\,x^{\alpha+\beta})\right)
=
\mu\,x^{\alpha+2\beta}
E_{\alpha,\,1+\frac{\beta}{\alpha},\,1+\frac{2\beta}{\alpha}}(\mu\,x^{\alpha+\beta})
+
\frac{\Gamma(\alpha+\beta+1)}{\Gamma(\beta+1)}\,x^{\beta}.
\]
Hence, for $\mu_k = \sqrt{\lambda_k}$, we compute
\[
\begin{aligned}
&D_x^{2\alpha}
\left(
-\frac{\Gamma(\beta+1)}{\Gamma(\alpha+\beta+1)}\,
\frac{g_k}{\mu_k}\,
x^{\alpha+\beta}
E_{\alpha,\,1+\frac{\beta}{\alpha},\,1+\frac{2\beta}{\alpha}}
\left(-\mu_k x^{\alpha+\beta}\right)
\right) \\
&\quad=
\frac{1}{x^{\beta}}\,\partial_{0,x}^{\alpha}
\left[
\frac{\Gamma(\beta+1)}{\Gamma(\alpha+\beta+1)}\,g_k\,
x^{\alpha+\beta}
E_{\alpha,\,1+\frac{\beta}{\alpha},\,1+\frac{2\beta}{\alpha}}
\left(-\mu_k x^{\alpha+\beta}\right)
- \frac{g_k}{\mu_k}
\right] \\
&\quad=
g_k
\left[
-\mu_k
\frac{\Gamma(\beta+1)}{\Gamma(\alpha+\beta+1)}\,
x^{\alpha+\beta}
E_{\alpha,\,1+\frac{\beta}{\alpha},\,1+\frac{2\beta}{\alpha}}
\left(-\mu_k x^{\alpha+\beta}\right)
+
1
\right] \\
&\quad=
\mu_k^2
\left(
-\frac{\Gamma(\beta+1)}{\Gamma(\alpha+\beta+1)}\,
\frac{g_k}{\mu_k}\,
x^{\alpha+\beta}
E_{\alpha,\,1+\frac{\beta}{\alpha},\,1+\frac{2\beta}{\alpha}}
\left(-\mu_k x^{\alpha+\beta}\right)
\right)
+ g_k.
\end{aligned}
\]
A similar computation yields
\[
\begin{aligned}
&D_x^{2\alpha}
\left(
\frac{\Gamma(\beta+1)}{\Gamma(\alpha+\beta+1)}\,
\frac{g_k}{\mu_k}\,
x^{\alpha+\beta}
E_{\alpha,\,1+\frac{\beta}{\alpha},\,1+\frac{2\beta}{\alpha}}
\left(\mu_k x^{\alpha+\beta}\right)
\right)\\
&\quad =
\mu_k^2
\left(
\frac{\Gamma(\beta+1)}{\Gamma(\alpha+\beta+1)}\,
\frac{g_k}{\mu_k}\,
x^{\alpha+\beta}
E_{\alpha,\,1+\frac{\beta}{\alpha},\,1+\frac{2\beta}{\alpha}}
\left(\mu_k x^{\alpha+\beta}\right)
\right) 
+ g_k.
\end{aligned}
\]
Therefore, a particular solution of the nonhomogeneous ODE \eqref{equ303} is given by
\[
u_k^*(x)
=
\frac{\Gamma(\beta+1)}{\Gamma(\alpha+\beta+1)}\,
\frac{g_k}{\mu_k}\,
x^{\alpha+\beta}
\left[
-d_1\,
E_{\alpha,\,1+\frac{\beta}{\alpha},\,1+\frac{2\beta}{\alpha}}
\left(-\mu_k x^{\alpha+\beta}\right)
+
d_2\,
E_{\alpha,\,1+\frac{\beta}{\alpha},\,1+\frac{2\beta}{\alpha}}
\left(\mu_k x^{\alpha+\beta}\right)
\right],
\]
where the constants $d_1$ and $d_2$ satisfy $d_1 + d_2 = 1$. 
To satisfy the boundedness condition 
\[
\lim_{x \to +\infty} |u_k(x)| < \infty,
\]
we note that
\[
E_{\alpha,\,1+\frac{\beta}{\alpha},\,1+\frac{2\beta}{\alpha}}
\left(\mu_k x^{\alpha+\beta}\right)
\to \infty \quad \text{as } x \to +\infty,
\]
while
\[
E_{\alpha,\,1+\frac{\beta}{\alpha},\,1+\frac{2\beta}{\alpha}}
\left(-\mu_k x^{\alpha+\beta}\right)
\to 0.
\]
Thus, boundedness requires $d_2 = 0$ and hence $d_1 = 1$.  
Therefore, a particular solution satisfying the boundary condition is given by
\[
u_k^*(x)
=
-\frac{\Gamma(\beta+1)\,g_k}{\Gamma(\alpha+\beta+1)\,\mu_k}
\,
x^{\alpha+\beta}
E_{\alpha,\,1+\frac{\beta}{\alpha},\,1+\frac{2\beta}{\alpha}}
\left(-\mu_k x^{\alpha+\beta}\right).
\]

In summary, combining the homogeneous and particular solutions, the solution of \eqref{equ303} can be written as
\[
u_k(x)
=
C_1\,
E_{\alpha,\,1+\frac{\beta}{\alpha},\,\frac{\beta}{\alpha}}
\left(-\mu_k x^{\alpha+\beta}\right)
+
C_2\,
E_{\alpha,\,1+\frac{\beta}{\alpha},\,\frac{\beta}{\alpha}}
\left(\mu_k x^{\alpha+\beta}\right)
+
u_k^*(x).
\]
To determine the constants $C_1$ and $C_2$, we use the boundary conditions
\[
u_k(0) = f_k,
\quad
\lim_{x \to +\infty} |u_k(x)| < \infty.
\]
Since
$
E_{\alpha,\,1+\frac{\beta}{\alpha},\,\frac{\beta}{\alpha}}
\left(\mu_k x^{\alpha+\beta}\right)
\to \infty
\quad \text{as } x \to +\infty,
$
boundedness implies $C_2 = 0$. Evaluating at $x = 0$ further yields $C_1 = f_k$. Therefore,
\[
u_k(x)
=
f_k\,
E_{\alpha,\,1+\frac{\beta}{\alpha},\,\frac{\beta}{\alpha}}
\left(-\mu_k x^{\alpha+\beta}\right)
-
\frac{\Gamma(\beta+1)}{\Gamma(\alpha+\beta+1)}
\frac{g_k}{\mu_k}\,
x^{\alpha+\beta}
E_{\alpha,\,1+\frac{\beta}{\alpha},\,1+\frac{2\beta}{\alpha}}
\left(-\mu_k x^{\alpha+\beta}\right).
\]
Substituting this expression into the Fourier expansion of $u(x,y)$, we obtain the solution of the direct problem:
\begin{align*}
u(x,y)
&=
\sum_{k=1}^{\infty}
f_k\,
E_{\alpha,\,1+\frac{\beta}{\alpha},\,\frac{\beta}{\alpha}}
\left(-\sqrt{\lambda_k}\,x^{\alpha+\beta}\right)
\varphi_k(y)
\\
&\quad
-
\sum_{k=1}^{\infty}
\frac{\Gamma(1+\beta)}{\Gamma(1+\alpha+\beta)}
\frac{g_k}{\sqrt{\lambda_k}}\,
x^{\alpha+\beta}
E_{\alpha,\,1+\frac{\beta}{\alpha},\,1+\frac{2\beta}{\alpha}}
\left(-\sqrt{\lambda_k}\,x^{\alpha+\beta}\right)
\varphi_k(y).
\end{align*}

To show the uniqueness of the solution, let $w_1(x,y)$ and $w_2(x,y)$ be two solutions of the forward problem \eqref{equa103} corresponding to the same data $f(y)$ and $g(y)$. 
Define 
\[
w(x,y) := w_1(x,y) - w_2(x,y).
\]
Then $w(x,y)$ satisfies the associated homogeneous problem
\[
\begin{cases}
D_x^{2\alpha} w(x,y) + L w(x,y) = 0, \quad & (x,y) \in (0,\infty) \times \Omega, \\
w(x,y) = 0, \quad & (x,y) \in (0,\infty) \times \partial \Omega, \\
w(0,y) = 0, \quad & y \in \Omega, \\
\displaystyle \lim_{x \to \infty} \|w(x,\cdot)\|_{L^2(\Omega)} = 0. &
\end{cases}
\]
Expanding $w(x,y)$ in terms of the eigenfunctions $\{\varphi_k\}_{k=1}^{\infty}$ of $-L$, we write
\[
w(x,y) = \sum_{k=1}^{\infty} w_k(x)\,\varphi_k(y), 
\quad 
w_k(x) = (w(x,\cdot), \varphi_k).
\]
Substituting this into the homogeneous equation and using orthogonality, we obtain for each $k \geq 1$:
\[
\begin{cases}
D_x^{2\alpha} w_k(x) - \lambda_k w_k(x) = 0, \\
w_k(0) = 0, \\
\displaystyle \lim_{x \to \infty} |w_k(x)| = 0.
\end{cases}
\]
The general solution of this homogeneous problem is given by
\[
w_k(x)
=
C_k\,E_{\alpha,\,1+\frac{\beta}{\alpha},\,\frac{\beta}{\alpha}}
\left(-\sqrt{\lambda_k}\,x^{\alpha+\beta}\right),
\]
for some constants $C_k$. 
Since $w_k(0)=0$ and $E_{\alpha,1+\frac{\beta}{\alpha},\frac{\beta}{\alpha}}(0)=1$, it follows that $C_k=0$ for all $k$. 
Therefore $w_k(x) \equiv 0$ and hence
\[
w(x,y) \equiv 0 \quad \text{for all } (x,y) \in (0,\infty) \times \Omega.
\]
This shows that $w_1(x,y)\! = \!w_2(x,y)$, and thus the solution of the direct problem \eqref{equa103} is unique.
\hfill $\Box$

\begin{theorem}\label{th02}
Let $f(y) \in D(-L)$ and $g(y) \in L^2(\Omega)$ satisfy the compatibility conditions. 
Then the solution $u(x,y)$ to the forward problem \eqref{equa103} enjoys the following regularity estimates:
\begin{equation}\label{equa310}
\begin{cases}
\displaystyle
\|u(x,\cdot)\|_{L^2(\Omega)}
\le 
\|f\|_{L^2(\Omega)} + \eta_4\,\|g\|_{L^2(\Omega)}, \\
\displaystyle
\sup_{x \in (0,\infty)}
\left\|
D_x^{2\alpha} u(x,\cdot)
\right\|_{L^2(\Omega)}
\le
\|f\|_{D(-L)}
+
\left(
1 + \frac{\eta_2}{\eta_3}
\right)
\|g\|_{L^2(\Omega)}, \\
\displaystyle
\sup_{x \in (0,\infty)}
\left\|
L u(x,\cdot)
\right\|_{L^2(\Omega)}
\le
\|f\|_{D(-L)}
+
\frac{\eta_2}{\eta_3}
\|g\|_{L^2(\Omega)}.
\end{cases}
\end{equation}
Here, 
$
\eta_4 
= \frac{\eta_2}{\eta_3 \lambda_1},
$
and $\lambda_1>0$ denotes the first eigenvalue of $-L$.
\end{theorem}

\textbf{Proof.}
Let $u(x,y)$ be decomposed into the sum of the homogeneous solution $u_1(x,y)$ and a particular solution $u_2(x,y)$ of the forward problem \eqref{equa103}, i.e.,
\[
u(x,y) = u_1(x,y) + u_2(x,y),
\]
where
\begin{equation}
\label{equ_u1u2}
\begin{cases}
\displaystyle
u_1(x,y)
=
\sum_{k=1}^{\infty}
f_k\,
E_{\alpha,\,1+\frac{\beta}{\alpha},\,\frac{\beta}{\alpha}}
\left(
-\sqrt{\lambda_k}\,x^{\alpha+\beta}
\right)
\varphi_k(y),\\
\displaystyle
u_2(x,y)
=
-
\sum_{k=1}^{\infty}
\frac{\Gamma(\beta+1)}{\Gamma(\alpha+\beta+1)}
\frac{g_k}{\sqrt{\lambda_k}}\,
x^{\alpha+\beta}
E_{\alpha,\,1+\frac{\beta}{\alpha},\,1+\frac{2\beta}{\alpha}}
\left(
-\sqrt{\lambda_k}\,x^{\alpha+\beta}
\right)
\varphi_k(y).
\end{cases}
\end{equation}
Here, $\{f_k\}$ and $\{g_k\}$ denote the Fourier coefficients of $f(y)$ and $g(y)$ with respect to the orthonormal basis $\{\varphi_k\}_{k=1}^{\infty}$ of $L^2(\Omega)$.

We now estimate the regularity of $u_1(x,y)$ and $u_2(x,y)$ separately.

For $u_1(x,y)$, by Parseval's identity and the fact that 
\[
\left|E_{\alpha,\,1+\frac{\beta}{\alpha},\,\frac{\beta}{\alpha}}
\left(-\sqrt{\lambda_k}x^{\alpha+\beta}\right)\right|\leq 1 \quad (\forall\,x>0,\;k\geq1),
\]
we obtain
\begin{align*}
\|u_1(x,\cdot)\|_{L^2(\Omega)}
&=
\left(
\sum_{k=1}^{\infty}
\left|
f_k\,E_{\alpha,\,1+\frac{\beta}{\alpha},\,\frac{\beta}{\alpha}}
\left(-\sqrt{\lambda_k}\,x^{\alpha+\beta}\right)
\right|^2
\right)^{1/2} \\
&\le 
\left(
\sum_{k=1}^{\infty}
|f_k|^2
\right)^{1/2}
=
\|f\|_{L^2(\Omega)}.
\end{align*}
For the particular solution $u_2(x,y)$, using the estimate 
\[
\left|
E_{\alpha,\,1+\frac{\beta}{\alpha},\,1+\frac{2\beta}{\alpha}}
\left(-\sqrt{\lambda_k}\,x^{\alpha+\beta}\right)
\right|
\le
\frac{1}{1+\eta_3\sqrt{\lambda_k}\,x^{\alpha+\beta}},
\]
we have
\begin{align*}
\|u_2(x,\cdot)\|_{L^2(\Omega)}
&=
\left(
\sum_{k=1}^{\infty}
\left|
-\frac{\Gamma(\beta+1)}{\Gamma(\alpha+\beta+1)}
\frac{g_k}{\sqrt{\lambda_k}}\,
x^{\alpha+\beta}
E_{\alpha,\,1+\frac{\beta}{\alpha},\,1+\frac{2\beta}{\alpha}}
\left(-\sqrt{\lambda_k}x^{\alpha+\beta}\right)
\right|^2
\right)^{1/2} \\
&\le
\left(
\sum_{k=1}^{\infty}
|g_k|^2\,
\left(
\frac{\Gamma(\beta+1)}{\Gamma(\alpha+\beta+1)}
\frac{x^{\alpha+\beta}}{\sqrt{\lambda_k}\,\bigl(1+\eta_3\sqrt{\lambda_k}x^{\alpha+\beta}\bigr)}
\right)^{\!2}
\right)^{1/2} \\
&\le
\left(
\sum_{k=1}^{\infty}
|g_k|^2\,
\left(
\frac{\Gamma(\beta+1)}{\Gamma(\alpha+\beta+1)}
\frac{1}{\eta_3\,\lambda_k}
\right)^{\!2}
\right)^{1/2} \\
&\le
\frac{\Gamma(\beta+1)}{\Gamma(\alpha+\beta+1)}
\frac{1}{\eta_3\,\lambda_1}
\left(
\sum_{k=1}^{\infty}
|g_n|^2
\right)^{1/2}
=
\eta_4\,\|g\|_{L^2(\Omega)},
\end{align*}
where $\eta_4 = \frac{\eta_2}{\eta_3 \lambda_1}$. Combining the above estimates yields
\[
\|u(x,\cdot)\|_{L^2(\Omega)} 
\le 
\|f\|_{L^2(\Omega)} + \eta_4\,\|g\|_{L^2(\Omega)}.
\]

It remains to estimate the norms of $\mathcal{D}_{x}^{2\alpha}u(x,y)$ and $-Lu(x,y)$. 
Using the eigenfunction expansion $u(x,y) = \sum_{k=1}^{\infty} u_k(x)\varphi_k(y)$ and the identities
\[
\mathcal{D}_{x}^{2\alpha} u_k(x) = \lambda_k u_k(x) + g_k, 
\quad 
-L\varphi_k = \lambda_k \varphi_k,
\]
we obtain
\begin{align*}
\mathcal{D}_{x}^{2\alpha}u(x,y)
&= 
\sum_{k=1}^{\infty}
\bigl(\lambda_k u_k(x) + g_k\bigr)\,\varphi_k(y) \\
&=
\sum_{k=1}^{\infty}
\lambda_k f_k
E_{\alpha,\,1+\frac{\beta}{\alpha},\,\frac{\beta}{\alpha}}
\left(-\sqrt{\lambda_k}\,x^{\alpha+\beta}\right)
\varphi_k(y) \\
&\quad+
\sum_{k=1}^{\infty}
g_k
\left[
\frac{\Gamma(\beta+1)}{\Gamma(\alpha+\beta+1)}
\bigl(-\sqrt{\lambda_k}\,x^{\alpha+\beta}\bigr)
E_{\alpha,\,1+\frac{\beta}{\alpha},\,1+\frac{2\beta}{\alpha}}
\left(-\sqrt{\lambda_k}x^{\alpha+\beta}\right)
+ 1
\right]
\varphi_k(y).
\end{align*}
Similarly,
\begin{align*}
-Lu(x,y)
&=
\sum_{k=1}^{\infty}
\lambda_k u_k(x)\,\varphi_k(y) \\
&=
\sum_{k=1}^{\infty}
\lambda_k f_k
E_{\alpha,\,1+\frac{\beta}{\alpha},\,\frac{\beta}{\alpha}}
\left(-\sqrt{\lambda_k}\,x^{\alpha+\beta}\right)\varphi_k(y)
\\
&\quad
-
\sum_{k=1}^{\infty}
g_k
\frac{\Gamma(\beta+1)}{\Gamma(\alpha+\beta+1)}
\sqrt{\lambda_k}\,x^{\alpha+\beta}
E_{\alpha,\,1+\frac{\beta}{\alpha},\,1+\frac{2\beta}{\alpha}}
\left(-\sqrt{\lambda_k}x^{\alpha+\beta}\right)
\varphi_k(y).
\end{align*}

By Parseval’s identity and using the bounds in Lemma \ref{lemma2}, we have
\[
\sup_{x\in(0,\infty)}
\|\mathcal{D}_{x}^{2\alpha}u(x,\cdot)\|_{L^2(\Omega)}
\le
\|f\|_{D(-L)}
+
\left(1 + \frac{\eta_2}{\eta_3}\right)
\|g\|_{L^2(\Omega)},
\]
and
\[
\sup_{x\in(0,\infty)}
\|Lu(x,\cdot)\|_{L^2(\Omega)}
\le
\|f\|_{D(-L)}
+
\frac{\eta_2}{\eta_3}
\|g\|_{L^2(\Omega)}.
\]

Combining these with the estimate of $\|u(x,\cdot)\|_{L^2(\Omega)}$ completes the proof of Theorem \ref{th02}. 
This also concludes the analysis of the well-posedness of the forward problem.
\hfill $\Box$.  

%==============================================
\section{Ill-posedness and Conditional Stability of the Inverse Problem} \label{sec:ill-posedness}
%==============================================

In the previous section, we have established the existence, uniqueness, and regularity of the solution to the forward problem \eqref{equa103}. 
We now turn our attention to the corresponding inverse source problem, namely the recovery of $g(y)$ from the interior measurement
\[
u(l,y)=h^\delta(y),\quad y\in\Omega,
\]
where $l>0$ is fixed, $h^\delta$ denotes the noisy observation of $h(y)=u(l,y)$, and the noise level $\delta>0$ satisfies 
$\|h^\delta-h\|_{L^2(\Omega)}\leq\delta$.

We now analyze the instability of the inverse source problem. 
Expanding the interior data $h(y) = u(l,y)$ in the eigenfunction basis $\{\varphi_k\}_{k=1}^\infty$ of $-L$, we obtain
\[
h(y)
=
\sum_{k=1}^{\infty}
h_k \, \varphi_k(y),
\quad
h_k
=
\bigl(h,\, \varphi_k\bigr).
\]
From the explicit representation of the solution to the forward problem, it follows that
\[
h_k
=
f_k\,
E_{\alpha,\,1+\frac{\beta}{\alpha},\,\frac{\beta}{\alpha}}
\left(
-\sqrt{\lambda_k}\,l^{\alpha+\beta}
\right)
-
\frac{\Gamma(1+\beta)}{\Gamma(1+\alpha+\beta)}
\frac{g_k}{\sqrt{\lambda_k}}\,
l^{\alpha+\beta}\,
E_{\alpha,\,1+\frac{\beta}{\alpha},\,1+\frac{2\beta}{\alpha}}
\left(
-\sqrt{\lambda_k}\,l^{\alpha+\beta}
\right).
\]
Solving formally for $g_k$ gives
\begin{equation}\label{equ:gkformal}
g_k
=
-
\frac{\Gamma(1+\alpha+\beta)}{\Gamma(1+\beta)}\,
\sqrt{\lambda_k}\,l^{-(\alpha+\beta)}\,
\frac{
h_k - f_k
E_{\alpha,\,1+\frac{\beta}{\alpha},\,\frac{\beta}{\alpha}}
\left(
-\sqrt{\lambda_k}\,l^{\alpha+\beta}
\right)
}{
E_{\alpha,\,1+\frac{\beta}{\alpha},\,1+\frac{2\beta}{\alpha}}
\left(
-\sqrt{\lambda_k}\,l^{\alpha+\beta}
\right)
}.
\end{equation}

For noisy data $h^\delta(y)$, satisfying $\|h^\delta - h\|_{L^2(\Omega)} \le \delta$, the corresponding coefficients become
\begin{equation}\label{equ:gkdeltaformal}
g_k^\delta
=
-
\frac{\Gamma(1+\alpha+\beta)}{\Gamma(1+\beta)}\,
\sqrt{\lambda_k}\,l^{-(\alpha+\beta)}\,
\frac{
h_k^\delta - f_k
E_{\alpha,\,1+\frac{\beta}{\alpha},\,\frac{\beta}{\alpha}}
\left(
-\sqrt{\lambda_k}\,l^{\alpha+\beta}
\right)
}{
E_{\alpha,\,1+\frac{\beta}{\alpha},\,1+\frac{2\beta}{\alpha}}
\left(
-\sqrt{\lambda_k}\,l^{\alpha+\beta}
\right)
}.
\end{equation}
From the estimate \eqref{equ208}, we know that
\[
\lim_{k\to\infty}
\sqrt{\lambda_k}\,l^{-(\alpha+\beta)}\,
\frac{1}{
E_{\alpha,\,1+\frac{\beta}{\alpha},\,1+\frac{2\beta}{\alpha}}
\left(
-\sqrt{\lambda_k}\,l^{\alpha+\beta}
\right)
}
= \infty,
\]
which implies that the inversion formula \eqref{equ:gkformal} and \eqref{equ:gkdeltaformal} is extremely sensitive to high-frequency components. 
Therefore, the reconstruction of $g(y)$ from noisy data $h^\delta(y)$ is an \emph{ill-posed} problem in the sense of Hadamard, as small perturbations in $h^\delta$ may lead to arbitrarily large errors in $g^\delta$.

Nevertheless, if additional a priori information is imposed on the source term $g(y)$—for example, assuming that it belongs to a suitable exponentially weighted space—then a conditional stability estimate can be established. 
This is stated as follows.

\begin{theorem}\label{the03}
Let $g(y)\in D\!\left(\exp\left(\frac{(-L)^r}{2}\right)\right)$ for some $r>0$, or $g(y) \in D\!\left((-L)^{p/2}\right)$. 
Then the inversion of the source term $g(y)$ from the exact data $h(y)=u(l,y)$ satisfies the following conditional stability estimate:
\begin{equation}\label{equ411}
\|g\|_{L^2(\Omega)}
\leq 
C \left(\|g\|_{D((-L)^{p/2})}\right)^{\frac{2}{p+2}}
\left(\|h\|_{L^2(\Omega)} + \|f\|_{L^2(\Omega)}\right)^{\frac{p}{p+2}},
\end{equation}
where $p>0$ is arbitrary and $C>0$ depends on $\alpha,\beta,\eta_2,\eta_3$ and $p$ but not on $k,f,h,g$.
\end{theorem}

\textbf{Proof.}
Since 
$
g(y) \in D\!\left(\exp\left(\tfrac{(-L)^r}{2}\right)\right), \; r>0,
$
we have
\[
\sum_{k=1}^\infty \exp(\lambda_k^r)\,|g_k|^2 < \infty.
\]
Because $\lambda_k \to \infty$ and the exponential growth dominates any polynomial rate, it follows that
\[
\lambda_k^{p} \ll \exp(\lambda_k^r), \quad k\to\infty,
\]
for any $p>0$. Hence,
\[
\sum_{k=1}^\infty \lambda_k^{p}\,|g_k|^2 < \infty 
\quad \Rightarrow \quad
g \in D\!\left((-L)^{p/2}\right).
\]

Next, from the inversion formula and the lower bound \eqref{equ208} for the Mittag-Leffler function, we obtain
\begin{align*}
|g_k|
&=
\left|
\frac{\Gamma(1+\alpha+\beta)}{\Gamma(1+\beta)}\,
\sqrt{\lambda_k}\,l^{-(\alpha+\beta)}
\frac{
h_k - f_k\,
E_{\alpha,\,1+\frac{\beta}{\alpha},\,\frac{\beta}{\alpha}}
\left(-\sqrt{\lambda_k}\,l^{\alpha+\beta}\right)
}{
E_{\alpha,\,1+\frac{\beta}{\alpha},\,1+\frac{2\beta}{\alpha}}
\left(-\sqrt{\lambda_k}\,l^{\alpha+\beta}\right)
}
\right| \\
&\le
C\,\lambda_k
\left|
h_k - f_k\,
E_{\alpha,\,1+\frac{\beta}{\alpha},\,\frac{\beta}{\alpha}}
\left(-\sqrt{\lambda_k}l^{\alpha+\beta}\right)
\right|,
\end{align*}
where $C>0$ is independent of $k$. Therefore,
\begin{align*}
\|g\|_{L^2(\Omega)}^2
&= \sum_{k=1}^\infty |g_k|^2 \le C^2 \sum_{k=1}^\infty \lambda_k^{2}
\left|h_k - f_k
E_{\alpha,\,1+\frac{\beta}{\alpha},\,\frac{\beta}{\alpha}}
\left(-\sqrt{\lambda_k}\,l^{\alpha+\beta}\right)\right|^2.
\end{align*}
Applying Hölder's inequality with exponents $\frac{p+2}{2}$ and $\frac{p+2}{p}$ yields
\begin{align*}
\|g\|_{L^2(\Omega)}^2
&\le 
C^2
\left(\sum_{k=1}^\infty \lambda_k^{p+2}
\left|h_k - f_k
E_{\alpha,\,1+\frac{\beta}{\alpha},\,\frac{\beta}{\alpha}}
\left(-\sqrt{\lambda_k}l^{\alpha+\beta}\right)
\right|^2
\right)^{\frac{2}{p+2}} \\
&\qquad \times
\left(
\sum_{k=1}^\infty 
\left|h_k - f_k
E_{\alpha,\,1+\frac{\beta}{\alpha},\,\frac{\beta}{\alpha}}
\left(-\sqrt{\lambda_k}l^{\alpha+\beta}\right)\right|^2
\right)^{\frac{p}{p+2}}.
\end{align*}
Using \eqref{equ203}, \eqref{equ208} and the definition \eqref{equ:gkformal} of $g_k$, we further obtain
\[
\|g\|_{L^2(\Omega)}^2
\le 
C^2 \frac{\eta_2^2}{\eta_3^2}
\left(\sum_{k=1}^\infty \lambda_k^{p}|g_k|^2\right)^{\frac{2}{p+2}}
\left(\|h\|_{L^2(\Omega)}^2 + \|f\|_{L^2(\Omega)}^2\right)^{\frac{p}{p+2}}.
\]
Taking square roots completes the proof of \eqref{equ411}. 
\hfill$\Box$

\medskip

The conditional stability in Theorem \ref{the03} indicates that the inversion of the source term $g(y)$ is stable only if an a priori bound is imposed, namely $g(y)\in D\!\left(\exp\left(\frac{(-L)^r}{2}\right)\right)$ or $g(y) \in D\!\left((-L)^{p/2}\right)$. 
This space reflects extremely high regularity, meaning that the data $h(y)$ must also be very smooth. 
In this case, a Hölder-type (almost Lipschitz as $p\to\infty$) conditional stability is achieved. However, in practical inverse problems the measured data $h^\delta(y)$ are contaminated with noise and generally do not satisfy such an a priori condition. 
Therefore, directly inverting $g$ from noisy data is unstable, and one must employ suitable regularization techniques to obtain a stable and convergent approximation of the source term.

%==============================================
\section{Exponential-type  Regularization Methods for Source Inversion}
\label{sec:Regularization}
%==============================================

Define
\[
\hat{h}^\delta(y)
=
\sum_{k=1}^{\infty}
f_k\,
E_{\alpha,\,1+\frac{\beta}{\alpha},\,\frac{\beta}{\alpha}}
\!\left(-\sqrt{\lambda_k}\,l^{\alpha+\beta}\right)
\varphi_k(y)
- h^\delta(y),
\]
where $h^\delta(y)$ denotes the noisy observation of $h(y)=u(l,y)$. 
Then the inverse source problem can be reformulated as the operator equation
\begin{equation}\label{equ501}
(Kg)(y) = \hat{h}^\delta(y), \qquad y \in \Omega,
\end{equation}
where
\[
(Kg)(y)
=
\sum_{k=1}^{\infty}
\frac{\Gamma(1+\beta)}{\Gamma(1+\alpha+\beta)}
\frac{g_k}{\sqrt{\lambda_k}}\,
l^{\alpha+\beta}
E_{\alpha,\,1+\frac{\beta}{\alpha},\,1+\frac{2\beta}{\alpha}}
\!\left(-\sqrt{\lambda_k}\,l^{\alpha+\beta}\right)
\varphi_k(y),
\]
with
\[
f_k = (f,\varphi_k), 
\quad 
g_k = (g,\varphi_k),
\quad k=1,2,\dots .
\]

It is straightforward to verify that $K : L^2(\Omega) \to L^2(\Omega)$ is a compact, self-adjoint operator. Hence, \eqref{equ501} is a Fredholm integral equation of the first kind. Its singular values are given by
\[
\sigma_k 
=
\frac{\Gamma(1+\beta)}{\Gamma(1+\alpha+\beta)}\,
\frac{1}{\sqrt{\lambda_k}}\,
l^{\alpha+\beta}
E_{\alpha,\,1+\frac{\beta}{\alpha},\,1+\frac{2\beta}{\alpha}}
\!\left(-\sqrt{\lambda_k}\,l^{\alpha+\beta}\right).
\]
Since $\sigma_k \to 0$ as $k \to \infty$, the inversion of $K$ is unstable, and therefore the problem is ill-posed.

To overcome the ill-posedness of the problem, it is necessary to employ appropriate regularization techniques. Recently, the exponential-type Tikhonov regularization method was proposed in \cite{ref16}. Motivated by this development, we adopt an exponential-type Tikhonov regularization strategy, which has proven to be an effective tool for stabilizing ill-posed inverse problems. Accordingly, the inverse problem is reformulated as the minimization of the following Tikhonov functional:
\begin{equation}\label{equ502}
J_{\tau,r}(g)
=
\|Kg - \hat{h}^\delta\|_{L^2(\Omega)}^2
+
\tau\,\|g\|_{r,\exp}^2,
\end{equation}
where $r \in \mathbb{R}$ is a fixed parameter, $\tau>0$ is the regularization parameter, and $\delta>0$ denotes the noise level in the measurement data. 
Here, $\|\cdot\|_{r,\exp}$ denotes the exponential-type norm associated with the operator $\exp\!\left(\frac{(-L)^r}{2}\right)$.

By taking the first variation of the functional $J_{\tau,r}(g)$ and applying the first-order optimality condition, we obtain the Euler equation
\begin{equation}\label{equ503}
\bigl(K^*K + \tau\,\exp((-L)^r)\bigr)\,g
= K^*\hat{h}^\delta.
\end{equation}
The solution of \eqref{equ503} coincides with the minimizer of the functional $J_{\tau,r}(g)$, that is, the regularized solution of the inverse problem. Exploiting the singular system $\{\sigma_k,\varphi_k\}_{k=1}^\infty$ of the compact self-adjoint operator $K$, the regularized solution admits the spectral representation
\begin{equation}\label{equ504}
g_{\tau,r}^{\delta}(y)
=
\sum_{k=1}^{\infty}
\frac{
\frac{\eta_{2}}{\sqrt{\lambda_k}}\,l^{\alpha+\beta}
E_{\alpha,\,1+\frac{\beta}{\alpha},\,1+\frac{2\beta}{\alpha}}
\!\left(-\sqrt{\lambda_k}\,l^{\alpha+\beta}\right)
}{
\left(
\frac{\eta_{2}}{\sqrt{\lambda_k}}\,l^{\alpha+\beta}
E_{\alpha,\,1+\frac{\beta}{\alpha},\,1+\frac{2\beta}{\alpha}}
\!\left(-\sqrt{\lambda_k}\,l^{\alpha+\beta}\right)
\right)^2
+ \tau\,\exp(\lambda_k^r)
}
\,\hat{h}_k^\delta\,\varphi_k(y),
\end{equation}
where $\hat{h}_k^\delta = (\hat{h}^\delta,\varphi_k)$. For exact (noise-free) data $h(y)$, the corresponding regularized solution is given by
\begin{equation}\label{equ505}
g_{\tau,r}(y)
=
\sum_{k=1}^{\infty}
\frac{
\frac{\eta_{2}}{\sqrt{\lambda_k}}\,l^{\alpha+\beta}
E_{\alpha,\,1+\frac{\beta}{\alpha},\,1+\frac{2\beta}{\alpha}}
\!\left(-\sqrt{\lambda_k}\,l^{\alpha+\beta}\right)
}{
\left(
\frac{\eta_{2}}{\sqrt{\lambda_k}}\,l^{\alpha+\beta}
E_{\alpha,\,1+\frac{\beta}{\alpha},\,1+\frac{2\beta}{\alpha}}
\!\left(-\sqrt{\lambda_k}\,l^{\alpha+\beta}\right)
\right)^2
+ \tau\,\exp(\lambda_k^r)
}
\,\hat{h}_k\,\varphi_k(y),
\end{equation}
with $\hat{h}_k = (\hat{h},\varphi_k)$.

Motivated by the general quasi-boundary value regularization idea proposed in \cite{Hao2019,Wei2022,ref20} and the above exponential-type Tikhonov regularization approach, we further propose a new quasi-boundary value regularization method, referred to as the exponential-type quasi-boundary value regularization method. The corresponding regularized problem is formulated as follows:
\begin{equation}\label{equa506_new}
\begin{cases}
D_x^{2\alpha}u(x,y)+Lu(x,y)=g(y), & (x,y)\in(0,\infty)\times\Omega,\\
u(x,y) = 0, & (x,y)\in(0,\infty)\times\partial\Omega,\\
u(0,y) = 0, & y\in\Omega,\\
\displaystyle \lim\limits_{x\to\infty}\|u(x,\cdot)\|_{L^2(\Omega)}<\infty, &\\
u(l,y)+\tau\,\exp\!\bigl((-L)^r\bigr)g(y)=\hat{h}^\delta(y),
\end{cases}
\end{equation}
where $r \in \mathbb{R}$ is a fixed parameter, $\tau>0$ is the regularization parameter, and $\delta>0$ denotes the noise level in the measurement data. By applying the spectral decomposition associated with the operator $L$, the regularized solution to \eqref{equa506_new} can be explicitly represented as
\begin{equation}\label{equ5007_new}
\tilde{g}_{\tau,r}^{\delta}(y)
=
\sum_{k=1}^{\infty}
\frac{
1
}{
\frac{\eta_{2}}{\sqrt{\lambda_k}}\,l^{\alpha+\beta}
E_{\alpha,\,1+\frac{\beta}{\alpha},\,1+\frac{2\beta}{\alpha}}
\!\left(-\sqrt{\lambda_k}\,l^{\alpha+\beta}\right)
+ \tau\,\exp(\lambda_k^r)
}
\,\hat{h}_k^\delta\,\varphi_k(y),
\end{equation}
where $\tilde{g}_{\tau,r}^{\delta}$ denotes the reconstructed source term obtained by the exponential-type quasi-boundary value regularization method. In the noise-free case, i.e., when $\hat{h}^\delta=\hat{h}$, the corresponding regularized solution is given by
\begin{equation}\label{equ5008_new}
\tilde{g}_{\tau,r}(y)
=
\sum_{k=1}^{\infty}
\frac{
1
}{
\frac{\eta_{2}}{\sqrt{\lambda_k}}\,l^{\alpha+\beta}
E_{\alpha,\,1+\frac{\beta}{\alpha},\,1+\frac{2\beta}{\alpha}}
\!\left(-\sqrt{\lambda_k}\,l^{\alpha+\beta}\right)
+ \tau\,\exp(\lambda_k^r)
}
\,\hat{h}_k\,\varphi_k(y). 
\end{equation}
%where $\hat{h}_k = (\hat{h},\varphi_k)$.

According to the idea of exponential-type quasi-boundary value regularization, the exponential-type Tikhonov regularization is equivalent to the following quasi-boundary value regularization problem:
\begin{equation}\label{equa506_new02}
\begin{cases}
D_x^{2\alpha}u(x,y)+Lu(x,y)=g(y), & (x,y)\in(0,\infty)\times\Omega,\\
u(x,y) = 0, & (x,y)\in(0,\infty)\times\partial\Omega,\\
u(0,y) = 0, & y\in\Omega,\\
\displaystyle \lim\limits_{x\to\infty}\|u(x,\cdot)\|_{L^2(\Omega)}<\infty, &\\
K^* u(l,y)+\tau\,\exp\!\bigl((-L)^r\bigr)g(y)=K^* \hat{h}^\delta(y),
\end{cases}
\end{equation}
where $K^*$ denotes the adjoint operator of $K$. In this formulation, the additional condition \eqref{equa104} is replaced by its adjoint counterpart, namely,
\begin{equation}\label{adjCond}
K^* u(l,y)=K^* \hat{h}^\delta(y).
\end{equation}
This formulation reveals the intrinsic connection between the exponential-type Tikhonov regularization and the quasi-boundary value method. In fact, the operator equation \eqref{equ503} coincides with the boundary condition imposed in \eqref{equa506_new02}. Therefore, the exponential-type Tikhonov regularization can be equivalently interpreted as a quasi-boundary value regularization approach, in which the regularization term is incorporated into the boundary condition in an adjoint sense.

\subsection{Convergence Analysis of Exponential-type Tikhonov Regularization}

In this subsection, we establish convergence estimates for the exponential-type Tikhonov regularization method under both polynomial-type and exponential-type source conditions. The results reveal how the choice of the parameter $r$ and the regularity of the exact solution influence the convergence rate.

\begin{theorem}\label{the07}
Let $g_{\tau,r}^{\delta}$ be the regularized solution of \eqref{equ503} obtained from $\hat{h}^{\delta}(y)$.

(i) If $r \le 0$ and the exact source term $g \in D((-L)^p)$ for some $p>0$, then
\[
\|g_{\tau,r}^{\delta} - g\| \le
\begin{cases}
C\left(1 + \|g\|_{D((-L)^p)}\right)\delta^{\frac{2}{3}}, & p \ge 2,\ \tau=\delta^{\frac{2}{3}}, \\[6pt]
C\left(1 + \|g\|_{D((-L)^p)}\right)\delta^{\frac{p}{p+1}}, & 0<p<2,\ \tau=\delta^{\frac{2}{p+1}} .
\end{cases}
\]

(ii) If $r>0$ and the exact source term $g \in D\!\left(\exp\left(\frac{(-L)^r}{2}\right)\right)$, then taking $\tau=\delta^{\frac{(p+1)(2+r)}{(2+r)+(p+1)}}$ yields
\[
\|g_{\tau,r}^{\delta} - g\|
\le 
C\!\left(1+\|g\|_{r,\mathrm{exp}}\right)
\delta^{\frac{2+r}{(2+r)+(p+1)}}
\]
where $p>1$ is an arbitrary positive number. 
Here $C>0$ is a generic constant depending only on $\alpha,\beta,l$ and the operator $L$, but independent of $g$, $\tau$ and $\delta$.
\end{theorem}

\textbf{Proof}. The proof follows similar arguments to those used in Theorem 3.1 of \cite{ref16}. For completeness, the detailed proof is provided in the Appendix. \hfill $\Box$

\begin{remark}
For the estimate in Case (ii), namely
\[
\|g_{\tau,r}^{\delta}-g\|
\;\leq\; C \,\bigl(1+\|g\|_{r,\mathrm{exp}}\bigr)
\delta^{\,\frac{2+r}{(2+r)+(p+1)}},
\qquad p>1,\ r>0,
\]
the asymptotic behaviors of the convergence rate can be summarized as follows:
\begin{itemize}
\item[$\diamond$] \textbf{As $p\to1^+$ (with $r>0$ fixed):}
$\frac{2+r}{(2+r)+(p+1)}\!\to\!\frac{2+r}{4+r}$,
yielding a rate of $\delta^{(2+r)/(4+r)}$.
\item[$\diamond$] \textbf{As $r\to0^+$:}
$\frac{2+r}{(2+r)+(p+1)}\!\to\!\frac{2}{p+3}$,
giving $\delta^{2/(p+3)}$.
\item[$\diamond$] \textbf{As $r\to+\infty$:}
$\frac{2+r}{(2+r)+(p+1)}\!\to\!1$,
leading to nearly linear convergence $\delta$.
\end{itemize}
\noindent
In Theorem 3.1(ii) of \cite{ref16}, the exponential regularization with $\gamma>0$ yields the convergence rate
\[
\|f_{\mu,\gamma}^{\delta}-f\|\le (C_3+C_4M_2)\,\delta^{\frac{2+\gamma}{6+\gamma}},
\]
which depends solely on the exponential parameter $\gamma$.
In contrast, our result reveals a coupled dependence on both the independent index $p$
and the exponential order $r$. This interplay leads to a refined spectrum of convergence rates, ranging from $\delta^{2/(p+3)}$ (as $r\to0$) to nearly linear convergence $\delta^{1}$ (as $r\to+\infty$). Consequently, the present result provides a more flexible and potentially sharper characterization of convergence behavior than that given in Theorem 3.1(ii) of \cite{ref16}. 
\end{remark}

Theorem \ref{the07} establishes the convergence rate of the regularized solution under an \emph{a priori} parameter choice strategy. 
In practice, however, obtaining an appropriate regularized solution through such a strategy requires fairly accurate prior knowledge of the smoothness of $g(y)$—in particular, reliable estimates of the norms 
$\|g\|_{D((-L)^p)}$ and $\|g\|_{r,\mathrm{exp}}$. 
This information is typically unavailable or very difficult to evaluate in real-world inverse problems. 
Therefore, \emph{a posteriori} parameter choice rules are commonly adopted in practical computations. 
Among these, the Morozov discrepancy principle \cite{ref18,ref19} is one of the most widely used and effective strategies. 
In what follows, we analyze the convergence behavior of the exponential-type Tikhonov regularization when the regularization parameter is determined according to the Morozov discrepancy principle.

The Morozov discrepancy principle prescribes choosing the regularization parameter $\tau$ such that it satisfies
\begin{equation}\label{equ509}
\|K g_{\tau,r}^{\delta} - \hat{h}^{\delta}\| = \rho\delta,
\end{equation}
where $\rho>1$ is a prescribed constant. When $0 < \rho\delta < \|\hat{h}^{\delta}\|$, 
it can be readily shown that the discrepancy equation \eqref{equ509} admits a unique solution for $\tau$. 
The proof of the existence and uniqueness of the solution to \eqref{equ509} follows standard arguments (Lemma 3.1) in \cite{ref16} and is therefore omitted here for brevity.

\begin{theorem}\label{the08}
Let $\hat h^\delta$ be the noisy data of $\hat h$ satisfying
\[
\|\hat h^\delta-\hat h\|\le \delta,
\qquad 
0<\rho\delta<\|\hat h^\delta\|,
\]
where $\rho>1$. Let $g_{\tau,r}^{\delta}$ be the regularized solution obtained by the exponential-type Tikhonov regularization method, where $\tau$ is chosen according to the Morozov discrepancy principle \eqref{equ509}.

\textup{(i)} If $r\le0$ and there exists a constant $C_5>0$ such that
\[
\|g\|_{D((-L)^p)}\le C_5,
\qquad p>0,
\]
then
\[
\|g_{\tau,r}^{\delta}-g\|
\le
\begin{cases}
C\!\left[
\left(\frac{C_5}{\rho-1}\right)^{\frac12}
+
C_5^{\frac1{p+1}}(1+\rho)^{\frac{p}{p+1}}
\right]\delta^{\frac12},
& p\ge1,\\
C\!\left[
\left(\frac{C_5}{\rho-1}\right)^{\frac1{p+1}}
+
C_5^{\frac1{p+1}}(1+\rho)^{\frac{p}{p+1}}
\right]\delta^{\frac{p}{p+1}},
& 0<p<1.
\end{cases}
\]

\textup{(ii)} If $r>0$ and there exists a constant $C_6>0$ such that $\|g\|_{r,\mathrm{exp}}\le C_6$, 
then, for any $p>0$,
\[
\|g_{\tau,r}^{\delta}-g\|
\le 
C(2C_6)^{\frac{2}{p+2}}
(1+\rho)^{\frac{p}{p+2}}
\delta^{\frac{p}{p+2}}.
\]
Here $C>0$ is independent of $\delta$, $\tau$ and $g$, but may depend on $p,r$. Since $p>0$ is arbitrary, the convergence rate approaches linear order as $p\to\infty$.
\end{theorem}

\textbf{Proof}. The proof follows similar arguments to those used in Theorem 3.2 of \cite{ref16}. For completeness, the detailed proof is provided in the Appendix. \hfill $\Box$

\medskip

It is worth noting that the convergence result in Theorem \ref{the08} 
still relies on the \emph{a priori} assumption 
\[
g\in D\!\left(\exp\!\left(\frac{(-L)^r}{2}\right)\right).
\]
This assumption is useful for deriving theoretical convergence rates, but it is not required for implementing the regularization method itself. 
Indeed, for a general source term $g\in L^2(\Omega)$, one can first approximate the inverse problem in a finite-dimensional space spanned by the eigenfunctions of the elliptic operator.

Let
\[
X_N=\operatorname{span}\{\varphi_1,\varphi_2,\ldots,\varphi_N\}.
\]
For any $\zeta_1,\zeta_2>0$, choose $N$ sufficiently large such that
\[
g_N(y)=\sum_{k=1}^{N}g_k\varphi_k(y),
\qquad
f_N(y)=\sum_{k=1}^{N}f_k\varphi_k(y),
\]
satisfy
\[
\|g-g_N\|\le \zeta_1,
\qquad
\|f-f_N\|\le \zeta_2.
\]
Here $g_k=(g,\varphi_k)$ and $f_k=(f,\varphi_k)$.

Define the noisy finite-dimensional data by
\[
\hat h_N^\delta(y)
=
\sum_{k=1}^{N}
f_k
E_{\alpha,\,1+\frac{\beta}{\alpha},\,\frac{\beta}{\alpha}}
\!\left(-\sqrt{\lambda_k}l^{\alpha+\beta}\right)
\varphi_k(y)
-
h^\delta(y).
\]
Then the finite-dimensional regularized solution associated with the exponential-type Tikhonov method is given by
\begin{equation}\label{equ512}
g_{N,\tau,r}^{\delta}(y)
=
\sum_{k=1}^{N}
\frac{
\frac{\eta_{2}}{\sqrt{\lambda_k}}\,l^{\alpha+\beta}
E_{\alpha,\,1+\frac{\beta}{\alpha},\,1+\frac{2\beta}{\alpha}}
\!\left(-\sqrt{\lambda_k}\,l^{\alpha+\beta}\right)
}{
\left(\frac{\eta_{2}}{\sqrt{\lambda_k}}\,l^{\alpha+\beta}
E_{\alpha,\,1+\frac{\beta}{\alpha},\,1+\frac{2\beta}{\alpha}}
\!\left(-\sqrt{\lambda_k}\,l^{\alpha+\beta}\right)\right)^2+\tau\exp(\lambda_k^r)
}
\hat h_{N,k}^{\delta}\,\varphi_k(y),
\end{equation}
where
\[
\hat h_{N,k}^{\delta}
=
f_k
E_{\alpha,\,1+\frac{\beta}{\alpha},\,\frac{\beta}{\alpha}}
\!\left(-\sqrt{\lambda_k}l^{\alpha+\beta}\right)
-
h_k^\delta,
\qquad
h_k^\delta=(h^\delta,\varphi_k).
\]

\begin{theorem}\label{the09}
Let $g\in L^2(\Omega)$ and $f\in L^2(\Omega)$. 
Let $g_N$ and $f_N$ be the spectral truncations defined above, satisfying
\[
\|g-g_N\|\le \zeta_1,
\qquad
\|f-f_N\|\le \zeta_2.
\]
Assume that the regularization parameter $\tau$ is chosen by the Morozov discrepancy principle. 
Then, for any fixed $N$ and any $p>0$, the finite-dimensional regularized solution $g_{N,\tau,r}^{\delta}$ satisfies
\begin{equation}\label{equ513}
\|g_{N,\tau,r}^{\delta}-g\|
\le
C M_N^{\frac{2}{p+2}}
\left[
(1+\rho)^{\frac{p}{p+2}}\delta^{\frac{p}{p+2}}
+
\zeta_2^{\frac{p}{p+2}}
+
\zeta_1^{\frac{p}{p+2}}
\right]
+\zeta_1,
\end{equation}
where $M_N>0$ is a finite-dimensional bound for 
$\|g_{N,\tau,r}^{\delta}-g_N\|_{D((-L)^{p/2})}$, and $C>0$ is independent of $\delta$, $\rho$, $\zeta_1$ and $\zeta_2$.
\end{theorem}

\textbf{Proof}. 
Since $g_{N,\tau,r}^{\delta}$ and $g_N$ both belong to the finite-dimensional space $X_N$, all Hilbert scale norms are equivalent on $X_N$. Hence, for fixed $N$, there exists a constant $M_N>0$ such that
\[
\|g_{N,\tau,r}^{\delta}-g_N\|_{D((-L)^{p/2})}\le M_N.
\]

Next, by the definition of $\hat h_N^\delta$ and the exact relation $Kg=\hat h$, we have
\[
\begin{aligned}
\hat h_N^\delta-Kg_N
&=
(F_N-h^\delta)-Kg_N \\
&=
(F_N-F)+(h-h^\delta)+K(g-g_N),
\end{aligned}
\]
where
\[
F(y)
=
\sum_{k=1}^{\infty}
f_k
E_{\alpha,\,1+\frac{\beta}{\alpha},\,\frac{\beta}{\alpha}}
\!\left(-\sqrt{\lambda_k}l^{\alpha+\beta}\right)
\varphi_k(y),
\]
and
\[
F_N(y)
=
\sum_{k=1}^{N}
f_k
E_{\alpha,\,1+\frac{\beta}{\alpha},\,\frac{\beta}{\alpha}}
\!\left(-\sqrt{\lambda_k}l^{\alpha+\beta}\right)
\varphi_k(y).
\]
Therefore,
\[
\|\hat h_N^\delta-Kg_N\|
\le
\|F_N-F\|+\|h-h^\delta\|+\|K(g-g_N)\|.
\]
Using the boundedness of the Mittag-Leffler factor and the continuity of $K$, we obtain
\[
\|\hat h_N^\delta-Kg_N\|
\le
C(\zeta_2+\delta+\zeta_1).
\]

By the Morozov discrepancy principle
\[
\|K g_{N,\tau,r}^{\delta}-\hat h_N^\delta\|
=\rho\delta, 
\]
we obtain
\[
\begin{aligned}
\|K(g_{N,\tau,r}^{\delta}-g_N)\|
&\le
\|K g_{N,\tau,r}^{\delta}-\hat h_N^\delta\|
+
\|\hat h_N^\delta-Kg_N\|  \\
&\le
C\bigl((1+\rho)\delta+\zeta_2+\zeta_1\bigr).
\end{aligned}
\]

Applying the conditional stability estimate \eqref{equ411} to 
$g_{N,\tau,r}^{\delta}-g_N$, we get
\[
\|g_{N,\tau,r}^{\delta}-g_N\|
\le
C
\|g_{N,\tau,r}^{\delta}-g_N\|_{D((-L)^{p/2})}^{\frac{2}{p+2}}
\|K(g_{N,\tau,r}^{\delta}-g_N)\|^{\frac{p}{p+2}}.
\]
Thus,
\[
\|g_{N,\tau,r}^{\delta}-g_N\|
\le
C M_N^{\frac{2}{p+2}}
\bigl((1+\rho)\delta+\zeta_2+\zeta_1\bigr)^{\frac{p}{p+2}}.
\]
Since $0<\frac{p}{p+2}<1$, we have
\[
(a+b+c)^{\frac{p}{p+2}}
\le
a^{\frac{p}{p+2}}
+b^{\frac{p}{p+2}}
+c^{\frac{p}{p+2}},
\]
for $a,b,c\ge0$. Hence,
\[
\|g_{N,\tau,r}^{\delta}-g_N\|
\le
C M_N^{\frac{2}{p+2}}
\left[
(1+\rho)^{\frac{p}{p+2}}\delta^{\frac{p}{p+2}}
+
\zeta_2^{\frac{p}{p+2}}
+
\zeta_1^{\frac{p}{p+2}}
\right].
\]
Finally,
\[
\|g_{N,\tau,r}^{\delta}-g\|
\le
\|g_{N,\tau,r}^{\delta}-g_N\|
+
\|g_N-g\|,
\]
which gives \eqref{equ513}. 
\hfill $\Box$

\subsection{Convergence Analysis of Exponential-type Quasi-boundary Value Regularization}

Following the analysis for the exponential-type Tikhonov method, we now establish convergence estimates for the exponential-type quasi-boundary value regularization method.

\begin{theorem}\label{the07_new}
Let $\tilde{g}_{\tau,r}^{\delta}$ be the regularization solution of the quasi-boundary value problem \eqref{equa506_new}.

(i) \textbf{Case $r\le 0$.} 
Assume that the exact source $g \in D((-L)^p)$ for some $p>0$. Then
\[
\|\tilde{g}_{\tau,r}^{\delta} - g\|
\le
\begin{cases}
C\left(1 + \|g\|_{D((-L)^p)}\right)\delta^{\frac{1}{2}}, & p \ge 1,\ \tau=\delta^{\frac{1}{2}}, \\
C\left(1 + \|g\|_{D((-L)^p)}\right)\delta^{\frac{p}{p+1}}, & 0<p<1,\ \tau=\delta^{\frac{1}{p+1}} .
\end{cases}
\]

(ii) \textbf{Case $r>0$.} 
Assume that $g \in D\!\left(\exp\!\left(\frac{(-L)^r}{2}\right)\right)$. 
Then, by choosing $\tau=\delta^{\frac{(p+1)(1+r)}{p+r+2}}$,
we obtain 
\[
\|\tilde{g}_{\tau,r}^{\delta} - g\|
\le 
C\!\left(1+\|g\|_{r,\mathrm{exp}}\right)
\delta^{\frac{1+r}{p+r+2}},
\]
where $p>1$ is an arbitrary constant. Here $C>0$ denotes a generic constant depending only on $\alpha,\beta,l$ and the operator $L$, but independent of $\delta$, $\tau$, and $g$.
\end{theorem}

\textbf{proof}. Direct calculation yields
\begin{align*}
\left\|\tilde{g}_{\tau,r}^\delta(y)\!-\!\tilde{g}_{\tau,r}(y)\right\|&=\left\|\sum_{k=1}^\infty \frac{1} {\frac{\eta_{2}}{\sqrt{\lambda_k}}l^{\alpha+\beta}E_{\alpha,1+\frac{\beta}{\alpha}, 1+\frac{2\beta}{\alpha}}\left(-\sqrt{\lambda_k}l^{\alpha+\beta}\right)\!+\! \tau\exp(\lambda_k^r)}\left(\hat{h}_k^\delta\!-\!\hat{h}_k\right)\,\varphi_k(y)\right\|\\
&\leq\delta \sup_k\frac{1}{\frac{\eta_{2}}{\sqrt{\lambda_k}}l^{\alpha+\beta} E_{\alpha,1+\frac{\beta}{\alpha},1+\frac{2\beta}{\alpha}}\left(-\sqrt{\lambda_k}l^{\alpha+\beta} \right)+\tau\exp(\lambda_k^r)}.
\end{align*}
and 
\begin{align*}
\left\|\tilde{g}_{\tau,r}(y)\!-\!g(y)\right\|^2& =\left\|\sum_{k=1}^{\infty}\frac{1}{\frac{\eta_{2}}{\sqrt{\lambda_k}} l^{\alpha\!+\!\beta}E_{\alpha,1\!+\!\frac{\beta}{\alpha},1\!+\!\frac{2\beta}{\alpha}} \left(-\sqrt{\lambda_k}l^{\alpha\!+\!\beta}\right)\!+\!\tau\exp(\lambda_k^r)}\hat{h}_k\, \varphi_k(y)\!-\!\sum_{k=1}^{\infty} g_k\varphi_k(y)\right\|^2 \\
&=\left\|\sum_{k=1}^\infty\frac{\frac{\eta_{2}}{\sqrt{\lambda_k}}l^{\alpha\!+\!\beta} E_{\alpha,1\!+\!\frac{\beta}{\alpha},1\!+\!\frac{2\beta}{\alpha}}\left(-\sqrt{\lambda_k} l^{\alpha\!+\!\beta}\right)}{\frac{\eta_{2}}{\sqrt{\lambda_k}}l^{\alpha\!+\!\beta} E_{\alpha,1\!+\!\frac{\beta}{\alpha},1\!+\!\frac{2\beta}{\alpha}}\left(-\sqrt{\lambda_k} l^{\alpha\!+\!\beta}\right)\!+\!\tau\exp(\lambda_k^r)}g_k\varphi_k(y)\!-\!\sum_{k=1}^\infty g_k\varphi_k(y)\right\|^2 \\
%&=\left\|\sum_{k=1}^\infty\frac{\tau\exp(\lambda_k^r)}{\left(\frac{\eta_{2}}{\sqrt{\lambda_k}} l^{\alpha+\beta}E_{\alpha,1+\frac{\beta}{\alpha},1+\frac{2\beta}{\alpha}} \left(-\sqrt{\lambda_k}l^{\alpha+\beta}\right)\right)^2+\tau\exp(\lambda_k^r)}g_k\varphi_k(y)\right\|\\
&=\sum_{k=1}^\infty\left(\frac{\tau\exp(\lambda_k^r)}{\frac{\eta_{2}}{\sqrt{\lambda_k}} l^{\alpha+\beta}E_{\alpha,1+\frac{\beta}{\alpha},1+\frac{2\beta}{\alpha}} \left(-\sqrt{\lambda_k}l^{\alpha+\beta}\right)+\tau\exp(\lambda_k^r)}\right)^2\left|g_k\right|^2. 
\end{align*}

\textbf{(i) Case $r\le0$.} 
When $r\le0$, since $\lim\limits_{k\to\infty}\lambda_k=+\infty$, there exists a constant $C>0$ such that 
$1\le \exp(\lambda_k^r)\le C$ for all $k$. 
Using estimate \eqref{equ208} and absorbing all constants into $C$, we obtain
\begin{equation}\label{equ5308_new}
\|\tilde{g}_{\tau,r}^{\delta}-\tilde{g}_{\tau,r}\|
\leq 
\frac{\delta}{\tau}; 
\end{equation}
and for $p\geq 1$, we have
\begin{align*}
\|\tilde{g}_{\tau,r}- g\|^2
&=\sum_{k=1}^\infty
\left(
\frac{\tau\exp(\lambda_k^r)}
{
\frac{\eta_{2}}{\sqrt{\lambda_k}}l^{\alpha+\beta}
E_{\alpha,1+\frac{\beta}{\alpha},1+\frac{2\beta}{\alpha}}
\!\left(-\sqrt{\lambda_k}l^{\alpha+\beta}\right)
+\tau\exp(\lambda_k^r)}
\right)^{\!2}\! g_k^2 \\
&\leq 
\sum_{k=1}^\infty
\left(
\frac{C\,\tau\,\lambda_k^{1-p}}{1+\tau\lambda_k}
\right)^{\!2}\!
\lambda_k^{2p} g_k^2
\le 
C^2\,\tau^2\,\| g \|_{D((-L)^p)}^2.
\end{align*}	
Hence,
\[
\|\tilde{g}_{\tau,r}- g\|\le C\,\tau\,\| g \|_{D((-L)^p)}.
\]
Combining this with \eqref{equ5308_new} and choosing $\tau=\delta^{1/2}$ yields
\[
\|\tilde{g}_{\tau,r}^{\delta}- g\|
\le
C\left(1+\|g \|_{D((-L)^p)}\right)\delta^{1/2}.
\]

For $0<p<1$, note that
\[
\frac{C\,\tau\,\lambda_k^{1-p}}{1+\tau\lambda_k}
\le 
C\,\sup_{t>0}\frac{t^{1-p}}{1+\tau t}\,\tau
\le 
C\,\tau^{p}.
\]
Therefore,
\[
\|\tilde{g}_{\tau,r}-g\|\le C\,\tau^{p}\|g\|_{D((-L)^p)}.
\]
Combining with \eqref{equ5308_new} and taking $\tau=\delta^{1/(p+1)}$ gives
\[
\|\tilde{g}_{\tau,r}^{\delta}-g\|
\le
C\left(1+\|g\|_{D((-L)^p)}\right)\delta^{p/(p+1)}.
\]

\textbf{Case (ii): $r>0$.} 
Using estimate \eqref{equ208} and absorbing all constants into a generic constant $C>0$, we obtain
\begin{align*}
\sup_{k\ge1}\frac{1}
{\frac{\eta_{2}}{\sqrt{\lambda_k}}l^{\alpha+\beta}
E_{\alpha,1+\frac{\beta}{\alpha},1+\frac{2\beta}{\alpha}}
\!\left(-\sqrt{\lambda_k}l^{\alpha+\beta}\right)
+\tau\exp(\lambda_k^r)}
&\le 
C\,\sup_{k\ge1}\frac{\lambda_k}{1+\tau\lambda_k^{1+r}}
\le C\,\tau^{-1/(1+r)}.
\end{align*}	
Hence,
\begin{equation}\label{equ5009_new}
\|\tilde{g}_{\tau,r}^{\delta}-\tilde{g}_{\tau,r}\|
\le 
C\,\frac{\delta}{\tau^{1/(1+r)}}.
\end{equation}

Since $r>0$ and $g\in D\!\left(\exp\!\left(\frac{(-L)^r}{2}\right)\right)$,  
it follows from the spectral representation and estimate \eqref{equ208} that
\begin{align*}
\|\tilde{g}_{\tau,r}-g\|^2
&=\sum_{k=1}^\infty
\left(
\frac{\tau\exp(\lambda_k^r)}
{
\frac{\eta_{2}}{\sqrt{\lambda_k}}l^{\alpha+\beta}
E_{\alpha,1+\frac{\beta}{\alpha},1+\frac{2\beta}{\alpha}}
\!\left(-\sqrt{\lambda_k}l^{\alpha+\beta}\right)
+\tau\exp(\lambda_k^r)}
\right)^{\!2}\! |g_k|^2 \\[3pt]
&\le 
\sum_{k=1}^\infty
\left(
\frac{C\,\tau\,\lambda_k\exp(\lambda_k^r)}
{1+\tau\lambda_k\exp(\lambda_k^r)}
\right)^{\!2}\! |g_k|^2.
\end{align*}

Let $p>1$ be an arbitrary positive number. 
Noting that $\exp((p-1)\lambda_k^r/(p+1))$ grows faster than any power $\lambda_k^{2/(p+1)}$,  
we can bound the above by
\begin{align*}
\|\tilde{g}_{\tau,r}-g\|^2
&\leq
C^2
\!\sup_{k\ge1}\!\left(
\frac{\tau\bigl(\lambda_k\exp(\lambda_k^r)\bigr)^{\frac{p}{p+1}}}
{1+\tau\lambda_k\exp(\lambda_k^r)}
\right)^{\!2}
\sum_{k=1}^\infty 
\exp(\lambda_k^r)\,|g_k|^2.
\end{align*}
A straightforward calculus argument yields
\[
\sup_{t>0}\frac{\tau\,t^{\frac{p}{p+1}}}{1+\tau t}
\le
C\,\tau^{\frac{1}{p+1}}.
\]
Thus,
\[
\|\tilde{g}_{\tau,r}-g\|
\le
C\,\tau^{\frac{1}{p+1}}\,\|g\|_{r,\mathrm{exp}}.
\]

Combining the two estimates, we obtain
\[
\|\tilde{g}_{\tau,r}^{\delta}-g\|
\le
C\frac{\delta}{\tau^{1/(1+r)}}
+
C\tau^{1/(p+1)}\|g\|_{r,\mathrm{exp}}.
\]
Balancing the two terms gives
\[
\tau=\delta^{\frac{(p+1)(1+r)}{p+r+2}}.
\]
With this choice of $\tau$, we obtain
\[
\|\tilde{g}_{\tau,r}^{\delta}-g\|
\le
C\!\left(1+\|g\|_{r,\mathrm{exp}}\right)
\delta^{\frac{1+r}{p+r+2}}.
\]
\hfill$\Box$

After establishing the convergence rates under an a priori parameter choice, we now study the case where the regularization parameter is determined by the Morozov discrepancy principle. 
The following theorem gives the corresponding convergence estimate for the exponential-type quasi-boundary value regularization method. 

\begin{lemma}\label{lem_log_estimate}
Let $r>0$ and $0<\tau<1$. Then there exists a constant $C>0$, independent of $\tau$, such that
\[
\sup_{t>0}\frac{t}{1+\tau t e^{t^r}}
\le C(1+|\ln\tau|)^{1/r}.
\]
More precisely, if
\[
F(t)=\frac{t}{1+\tau t e^{t^r}},\qquad t>0,
\]
and $t_\tau$ is the maximum point of $F$, then
\[
\sup_{t>0}\frac{t}{1+\tau t e^{t^r}}
=
F(t_\tau)
\le t_\tau
\le C(1+|\ln\tau|)^{1/r}.
\]
\end{lemma}

\textbf{proof}. 
A direct calculation gives
\[
F'(t)
=
\frac{1-\tau r t^{r+1}e^{t^r}}
{\left(1+\tau t e^{t^r}\right)^2}.
\]
Since the function
\[
t\mapsto \tau r t^{r+1}e^{t^r}
\]
is strictly increasing from $0$ to $+\infty$ on $(0,\infty)$, the function $F$ has a unique maximum point $t_\tau>0$, which satisfies
\[
\tau r t_\tau^{r+1}e^{t_\tau^r}=1.
\]
Taking logarithms yields
\[
t_\tau^r+(r+1)\ln t_\tau
=
|\ln\tau|-\ln r .
\]
If $t_\tau\ge1$, then $\ln t_\tau\ge0$, and hence
\[
t_\tau^r
\le
|\ln\tau|-\ln r
\le
|\ln\tau|+|\ln r|
\le
C(1+|\ln\tau|).
\]
Therefore,
\[
t_\tau\le C(1+|\ln\tau|)^{1/r}.
\]
If $0<t_\tau<1$, the same estimate is trivial. Consequently,
\[
\sup_{t>0}\frac{t}{1+\tau t e^{t^r}}
=
F(t_\tau)
\le t_\tau
\le C(1+|\ln\tau|)^{1/r}.
\]
\hfill $\Box$

\begin{theorem}\label{the08_new}
Let $r>0$, and assume that the exact source term satisfies
\[
g\in D\!\left(\exp\!\left(\frac{(-L)^r}{2}\right)\right),
\qquad
\|g\|_{r,\mathrm{exp}}\le E .
\]
Let $\tilde g_{\tau,r}^{\delta}$ be the regularized solution obtained by the
exponential quasi-boundary value regularization method, where the regularization
parameter $\tau$ is chosen by the Morozov discrepancy principle \ref{equ509}. 
Then, for any $p>0$, the following estimate holds:
\[
\|\tilde g_{\tau,r}^{\delta}-g\|
\le
C\,\delta\bigl(1+|\ln\delta|\bigr)^{1/r}
+
C\,E^{\frac{2}{p+2}}(1+\rho)^{\frac{p}{p+2}}
\delta^{\frac{p}{p+2}} .
\]
In particular, since
\[
\delta\bigl(1+|\ln\delta|\bigr)^{1/r}
=
o\!\left(\delta^{\frac{p}{p+2}}\right),
\qquad \delta\to0^+,
\]
we further obtain
\[
\|\tilde g_{\tau,r}^{\delta}-g\|
\le
C(1+E)(1+\rho)^{\frac{p}{p+2}}
\delta^{\frac{p}{p+2}},
\qquad p>0.
\]
Here $C>0$ is a generic constant independent of $\delta$ and $\tau$.
\end{theorem}

\textbf{Proof}. From the proof of Theorem \ref{the07_new}, we know 
\begin{align*}
\left\|\tilde{g}_{\tau,r}^\delta(y)\!-\!\tilde{g}_{\tau,r}(y)\right\|&\leq\delta \sup_k\frac{1}{\frac{\eta_{2}}{\sqrt{\lambda_k}}l^{\alpha+\beta} E_{\alpha,1+\frac{\beta}{\alpha},1+\frac{2\beta}{\alpha}}\left(-\sqrt{\lambda_k}l^{\alpha+\beta} \right)+\tau\exp(\lambda_k^r)}.
\end{align*}
Using estimate \eqref{equ208} and absorbing all constants into a generic constant $C>0$, we obtain
\begin{align*}
\sup_{k\ge1}\frac{1}
{\frac{\eta_{2}}{\sqrt{\lambda_k}}l^{\alpha+\beta}
E_{\alpha,1+\frac{\beta}{\alpha},1+\frac{2\beta}{\alpha}}
\!\left(-\sqrt{\lambda_k}l^{\alpha+\beta}\right)
+\tau\exp(\lambda_k^r)}
&\le C
\sup_{k\ge1}
\frac{\lambda_k}{1+\tau\lambda_k\exp(\lambda_k^r)}\\
&\le
C\bigl(1+|\ln\tau|\bigr)^{1/r}.
\end{align*}	
Hence,
\[
\|\tilde g_{\tau,r}^{\delta}-\tilde g_{\tau,r}\|
\le
C\delta\bigl(1+|\ln\tau|\bigr)^{1/r}.
\]

Next, we show that $\tau$ is not too small. From the discrepancy principle, $\|\hat h^\delta-\hat h\|\le\delta$ and 
\[
0<\frac{\tau \exp(\lambda_k^r)}{\frac{\eta_{2}}{\sqrt{\lambda_k}}l^{\alpha+\beta}
E_{\alpha,1+\frac{\beta}{\alpha},1+\frac{2\beta}{\alpha}}
\!\left(-\sqrt{\lambda_k}l^{\alpha+\beta}\right)+\tau \exp(\lambda_k^r)}<1,
\] we obtain
\begin{align*}
\rho\delta
&=
\left\|
\sum_{k=1}^{\infty}
\frac{\tau \exp(\lambda_k^r)}{\frac{\eta_{2}}{\sqrt{\lambda_k}}l^{\alpha+\beta}
E_{\alpha,1+\frac{\beta}{\alpha},1+\frac{2\beta}{\alpha}}
\!\left(-\sqrt{\lambda_k}l^{\alpha+\beta}\right)+\tau \exp(\lambda_k^r)}
\hat h_k^\delta \varphi_k
\right\|\\
&\leq \delta + \left\|
\sum_{k=1}^{\infty}
\frac{\tau \exp(\lambda_k^r)}{\frac{\eta_{2}}{\sqrt{\lambda_k}}l^{\alpha+\beta}
E_{\alpha,1+\frac{\beta}{\alpha},1+\frac{2\beta}{\alpha}}
\!\left(-\sqrt{\lambda_k}l^{\alpha+\beta}\right)+\tau \exp(\lambda_k^r)}
\hat h_k \varphi_k
\right\|.
\end{align*}
Using the exact relation $\hat h_k=\frac{\eta_{2}}{\sqrt{\lambda_k}}l^{\alpha+\beta}
E_{\alpha,1+\frac{\beta}{\alpha},1+\frac{2\beta}{\alpha}}
\!\left(-\sqrt{\lambda_k}l^{\alpha+\beta}\right) g_k$, together with the inequality
\[
\left(\frac{ab}{a+b}\right)^2\le ab
\]
for $a,b>0$, we have
\begin{align*}
(\rho-1)^2\delta^2
&\le
\sum_{k=1}^{\infty}
\left(
\frac{\tau \exp(\lambda_k^r) \frac{\eta_{2}}{\sqrt{\lambda_k}}l^{\alpha+\beta}
E_{\alpha,1+\frac{\beta}{\alpha},1+\frac{2\beta}{\alpha}}
\!\left(-\sqrt{\lambda_k}l^{\alpha+\beta}\right)}
{\frac{\eta_{2}}{\sqrt{\lambda_k}}l^{\alpha+\beta}
E_{\alpha,1+\frac{\beta}{\alpha},1+\frac{2\beta}{\alpha}}
\!\left(-\sqrt{\lambda_k}l^{\alpha+\beta}\right)+\tau \exp(\lambda_k^r)}
\right)^2 |g_k|^2\\ 
&\le C\tau
\sum_{k=1}^{\infty}\exp(\lambda_k^r)|g_k|^2
\le
C\tau E^2 .
\end{align*}
Hence,
\[
\tau\ge C\frac{(\rho-1)^2}{E^2}\delta^2 .
\]
Consequently, for sufficiently large $C$, 
\[
1+|\ln\tau|
\le
C \bigl(1+|\ln\delta|\bigr),
\]
and therefore
\[
\|\tilde g_{\tau,r}^{\delta}-\tilde g_{\tau,r}\|
\le
C\delta\bigl(1+|\ln\delta|\bigr)^{1/r}.
\]

It remains to estimate the regularization bias. We have
\[
\tilde g_{\tau,r}-g
=
-\sum_{k=1}^{\infty}
\frac{\tau \exp(\lambda_k^r)}{\frac{\eta_{2}}{\sqrt{\lambda_k}}l^{\alpha+\beta}
E_{\alpha,1+\frac{\beta}{\alpha},1+\frac{2\beta}{\alpha}}
\!\left(-\sqrt{\lambda_k}l^{\alpha+\beta}\right)+\tau \exp(\lambda_k^r)}g_k\varphi_k .
\]
Since the multiplier is bounded by $1$,
\[
\|\tilde g_{\tau,r}-g\|_{D((-L)^{p/2})}
\le
\|g\|_{D((-L)^{p/2})}.
\]
Moreover, the exponential source condition implies
\[
\|g\|_{D((-L)^{p/2})}
\le
C\|g\|_{r,\mathrm{exp}}
\le
CE .
\]
On the other hand, 
\[
K(\tilde g_{\tau,r}-g)
=
-\sum_{k=1}^{\infty}
\frac{\tau \exp(\lambda_k^r)}{\frac{\eta_{2}}{\sqrt{\lambda_k}}l^{\alpha+\beta}
E_{\alpha,1+\frac{\beta}{\alpha},1+\frac{2\beta}{\alpha}}
\!\left(-\sqrt{\lambda_k}l^{\alpha+\beta}\right) + \tau \exp(\lambda_k^r)}
\hat h_k\varphi_k . 
\]
Using the discrepancy principle and the noise bound, we obtain
\[
\begin{aligned}
\|K(\tilde g_{\tau,r}-g)\|
&\le
\left\|
\sum_{k=1}^{\infty}
\frac{\tau \exp(\lambda_k^r)}{\frac{\eta_{2}}{\sqrt{\lambda_k}}l^{\alpha+\beta}
E_{\alpha,1+\frac{\beta}{\alpha},1+\frac{2\beta}{\alpha}}
\!\left(-\sqrt{\lambda_k}l^{\alpha+\beta}\right)+\tau \exp(\lambda_k^r)}
\hat h_k^\delta\varphi_k
\right\|
+ \\
&\quad 
\left\|
\sum_{k=1}^{\infty}
\frac{\tau \exp(\lambda_k^r)}{\frac{\eta_{2}}{\sqrt{\lambda_k}}l^{\alpha+\beta}
E_{\alpha,1+\frac{\beta}{\alpha},1+\frac{2\beta}{\alpha}}
\!\left(-\sqrt{\lambda_k}l^{\alpha+\beta}\right)+\tau \exp(\lambda_k^r)}
(\hat h_k-\hat h_k^\delta)\varphi_k
\right\|  \\
&\le
\rho\delta+\delta
=
(1+\rho)\delta .
\end{aligned}
\]
Applying the conditional stability estimate to $\tilde g_{\tau,r}-g$, 
we obtain
\[
\|\tilde g_{\tau,r}-g\|
\le
C
\|\tilde g_{\tau,r}-g\|_{D((-L)^{p/2})}^{\frac{2}{p+2}}
\|K(\tilde g_{\tau,r}-g)\|^{\frac{p}{p+2}}.
\]
Therefore,
\[
\|\tilde g_{\tau,r}-g\|
\le
C
E^{\frac{2}{p+2}}
(1+\rho)^{\frac{p}{p+2}}
\delta^{\frac{p}{p+2}}.
\]

Finally, by the triangle inequality,
\[
\|\tilde g_{\tau,r}^{\delta}-g\|
\le
\|\tilde g_{\tau,r}^{\delta}-\tilde g_{\tau,r}\|
+
\|\tilde g_{\tau,r}-g\|,
\]
which yields
\[
\|\tilde g_{\tau,r}^{\delta}-g\|
\le
C\delta\bigl(1+|\ln\delta|\bigr)^{1/r}
+
C
E^{\frac{2}{p+2}}
(1+\rho)^{\frac{p}{p+2}}
\delta^{\frac{p}{p+2}}.
\]
Since
\[
\delta\bigl(1+|\ln\delta|\bigr)^{1/r}
=
o\!\left(\delta^{\frac{p}{p+2}}\right)
\quad \text{as } \delta\to0^+,
\]
the simplified estimate follows. \hfill $\Box$

\begin{remark}
For the exponential quasi-boundary value regularization method, the case $r\le0$ is essentially different from the case $r>0$. Indeed, when $r\le0$, the factor $\exp(\lambda_k^r)$ remains bounded and therefore does not provide sufficient high-frequency damping. Consequently, under only the noise assumption
\[
\|\hat h^\delta-\hat h\|\le\delta,
\]
the Morozov discrepancy principle does not yield an effective a posteriori convergence estimate for the regularized solution. Additional regularity assumptions on the noisy data or a uniform boundedness condition on the regularized solutions in a stronger Hilbert scale would be required to recover such a convergence result.
\end{remark}

By applying the same spectral truncation procedure to the exponential quasi-boundary value regularization method, the corresponding finite-dimensional regularized solution is given by
\begin{equation}\label{equ5022_new}
\tilde{g}_{N,\tau,r}^{\delta}(y)
=
\sum_{k=1}^{N}
\frac{
1
}{
\frac{\eta_{2}}{\sqrt{\lambda_k}}\,l^{\alpha+\beta}
E_{\alpha,\,1+\frac{\beta}{\alpha},\,1+\frac{2\beta}{\alpha}}
\!\left(-\sqrt{\lambda_k}\,l^{\alpha+\beta}\right) + \tau\,\exp(\lambda_k^r)
}
\,\hat{h}_{N,k}^{\delta}\,\varphi_k(y),
\end{equation}
where
\[
\hat{h}_{N,k}^{\delta}
=
f_k
E_{\alpha,\,1+\frac{\beta}{\alpha},\,\frac{\beta}{\alpha}}
\!\left(-\sqrt{\lambda_k}\,l^{\alpha+\beta}\right)
-h_k^\delta,
\qquad
h_k^\delta=(h^\delta,\varphi_k).
\]

\begin{theorem}\label{the10}
Let $g\in L^2(\Omega)$ and $f\in L^2(\Omega)$, and denote $g_k=(g,\varphi_k)$ and $f_k=(f,\varphi_k)$. 
Then, for any $\zeta_1,\zeta_2>0$, there exist truncated approximations
\[
g_N(y)=\sum_{k=1}^{N}g_k\,\varphi_k(y), 
\quad 
f_N(y)=\sum_{k=1}^{N}f_k\,\varphi_k(y),
\]
such that 
\[
\|g-g_N\|\le\zeta_1,
\quad
\|f-f_N\|\le\zeta_2.
\]
Let $\tilde{g}_{N,\tau,r}^{\delta}$ be the finite-dimensional regularized solution defined by \eqref{equ5022_new}, where the regularization parameter $\tau$ is chosen by the Morozov discrepancy principle. Then, for any fixed $p>0$, the following estimate holds:
\begin{equation}\label{equ524_new}
\|\tilde{g}_{N,\tau,r}^{\delta}-g\|
\le
C_N
\left[
(1+\rho)^{\frac{p}{p+2}}\delta^{\frac{p}{p+2}}
+
\zeta_2^{\frac{p}{p+2}}
+
\zeta_1^{\frac{p}{p+2}}
\right]
+\zeta_1 .
\end{equation}
Here $C_N>0$ is a constant depending on $N$, $p$, $r$, and the exact data, but independent of $\delta$, $\tau$, $\rho$, $\zeta_1$, and $\zeta_2$.
\end{theorem}

\textbf{Proof}. 
The proof is analogous to that of Theorem \ref{the09}. Indeed, after replacing the exponential-type Tikhonov filter in \eqref{equ512} with the exponential quasi-boundary value filter in \eqref{equ5022_new}, the same finite-dimensional norm equivalence argument, the Morozov discrepancy principle, the conditional stability estimate \eqref{equ411}, and the triangle inequality lead directly to \eqref{equ524_new}. 
For brevity, the details are omitted. 
\hfill $\Box$

\begin{remark}
Theorems \ref{the09} and \ref{the10} show that both the exponential-type 
Tikhonov regularization method and the exponential quasi-boundary value 
regularization method remain applicable to general source terms 
$g\in L^2(\Omega)$ through finite-dimensional spectral approximation. 
In these finite-dimensional settings, no exponential-type source condition such as
\[
g\in D\!\left(\exp\!\left(\frac{(-L)^r}{2}\right)\right)
\]
is required for the exact source term. Moreover,  the corresponding finite-dimensional error estimates 
hold for all $r\in\mathbb{R}$, although the constants may depend on the truncation 
dimension $N$, the parameter $r$, and the Hilbert-scale index $p$. 
This indicates that the proposed regularization strategies can be stably implemented 
for low-regularity source terms in $L^2(\Omega)$ after a suitable spectral truncation.
\end{remark}

%==============================================
\section{Numerical Experiments} \label{sec:Numerical Experiments}
%==============================================	
This section presents several numerical examples to verify the effectiveness 
of the proposed exponential-type regularization methods in recovering 
the source term of a fractional-order elliptic equation 
of Tricomi-Gellerstedt-Keldysh type.

For simplicity, we consider the one-dimensional domain 
$\Omega=(0,\pi)$ and take
\[
L=\frac{\partial^2}{\partial y^2}:
H^2(0,\pi)\cap H_0^1(0,\pi)\to L^2(0,\pi).
\]
The corresponding eigenpairs of $-L$ are
\[
\lambda_k=k^2,
\qquad
\varphi_k(y)=\sqrt{\frac{2}{\pi}}\sin(ky),
\qquad k=1,2,\ldots .
\]

For the two regularization methods proposed in this paper, 
the corresponding numerical approximations of the regularized solutions 
are obtained by truncating the spectral series in 
\eqref{equ512} and \eqref{equ5022_new}, respectively. 
In the computations, the measurement location is fixed at $l=1$. 
For both the exponential-type Tikhonov regularized solution \eqref{equ512} 
and the exponential quasi-boundary value regularized solution \eqref{equ5022_new}, 
two choices of the exponential parameter, $r=-1$ and $r=0.2$, are tested. 
For all regularization methods considered below, including the generalized 
Tikhonov regularized solution \eqref{equ603}, the number of truncated terms is taken as 
$N=50$. In the generalized Tikhonov regularization method, we use $r=1$.

Since in most practical applications it is difficult, or even impossible, 
to estimate the norm of the exact source term, all numerical results 
in this section are obtained using the \emph{a posteriori} parameter 
selection strategy based on the Morozov discrepancy principle. 
The discrepancy parameter is fixed as $\rho=1.01$ throughout the computations.

The noisy data $h^{\delta}(y)$ are generated by adding random perturbations to the exact data $h(y)$ in the following relative form:
\begin{equation}\label{equ601}
h^\delta(y) = \bigl(1 + \varepsilon R\bigr)\,h(y),
\end{equation}
where $R$ is a random variable following the standard normal distribution, and $\varepsilon$ denotes the relative noise level in the discrete data. Consequently, the noise magnitude is given by $\delta = \|\varepsilon R h\|$.

To evaluate the effectiveness of the two exponential-type regularization methods proposed in this paper, we compare them with the classical Tikhonov regularization method and the generalized Tikhonov regularization method \cite{Ma2018,Can2020}. The generalized Tikhonov regularized solution is given by
\begin{equation}\label{equ603}
\hat{g}_{\tau,r}^\delta(y)
=
\sum_{k=1}^{N}
\frac{
\frac{\eta_{2}}{\sqrt{\lambda_k}}\,l^{\alpha+\beta}
E_{\alpha,\,1+\frac{\beta}{\alpha},\,1+\frac{2\beta}{\alpha}}
\!\left(-\sqrt{\lambda_k}\,l^{\alpha+\beta}\right)
}{
\left(
\frac{\eta_{2}}{\sqrt{\lambda_k}}\,l^{\alpha+\beta}
E_{\alpha,\,1+\frac{\beta}{\alpha},\,1+\frac{2\beta}{\alpha}}
\!\left(-\sqrt{\lambda_k}\,l^{\alpha+\beta}\right)
\right)^{2}
+\tau\,\lambda_k^r
}
\,\hat{h}_k^\delta\,\varphi_k(y).
\end{equation}
When $r=0$, since $\lambda_k^0=1$, \eqref{equ603} reduces to the classical Tikhonov regularization method. When $r>0$, \eqref{equ603} corresponds to the generalized Tikhonov regularization method; in the numerical comparisons below, we take $r=1$ for this method.

\textbf{Example 1.}
Let the exact boundary value be
\[
f(y)=\frac{\sqrt{3}}{3\pi}\bigl(y^3-1\bigr)\sin(3y),
\]
and let the exact source term be
\[
g(y)=y(y-1)(y-\pi)\cos y,\qquad 0\le y\le \pi .
\]

Figure \ref{fig01}--\ref{fig011} displays the exact source term and the reconstructed results obtained by the two proposed exponential-type regularization methods. The relative errors of the exponential-type Tikhonov regularization method, the exponential quasi-boundary value regularization method, the classical Tikhonov regularization method, and the generalized Tikhonov regularization method are compared in Table \ref{tab01}. Since the numerical results for different pairs of $(\alpha,\beta)$ show similar behavior, Table \ref{tab01} reports only the relative errors for $(\alpha,\beta)=(0.9,0.6)$.

\begin{figure}[H] 
\centering
\subfloat[ $\alpha=0.7,\beta=0.3,r=-1$]{\includegraphics[width=0.35\textwidth]{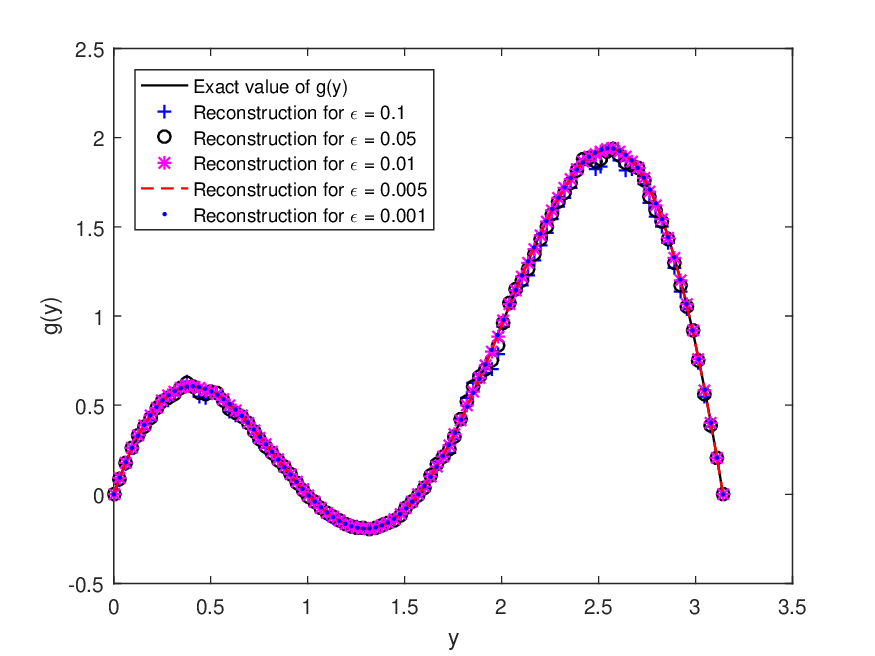}\label{fig01:sub1}}
\subfloat[ $\alpha=0.9,\beta=0.6,r=-1$]{\includegraphics[width=0.35\textwidth]{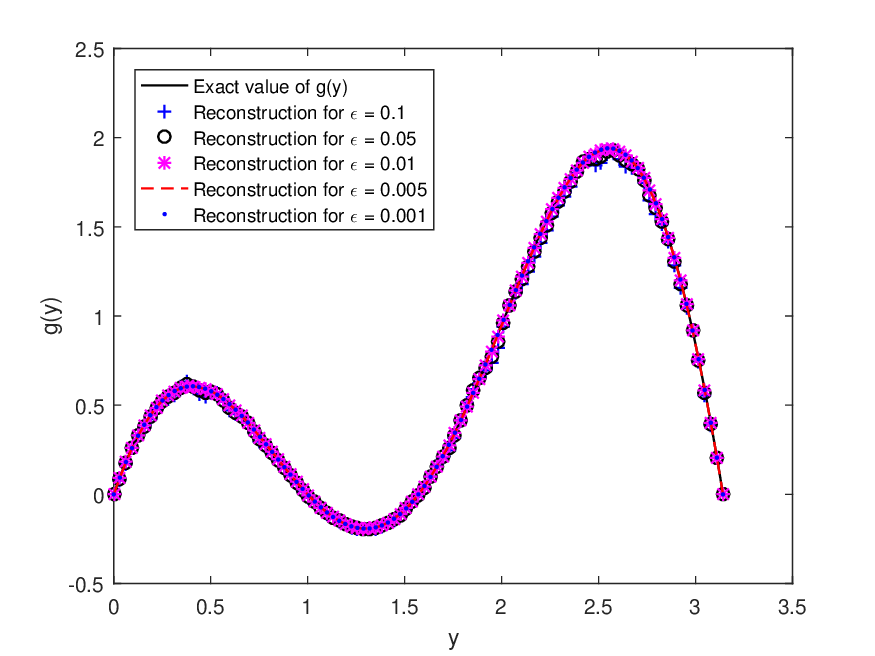}\label{fig01:sub2}}
\\
\subfloat[ $\alpha=0.7,\beta=0.3,r=0.2$]{\includegraphics[width=0.35\textwidth]{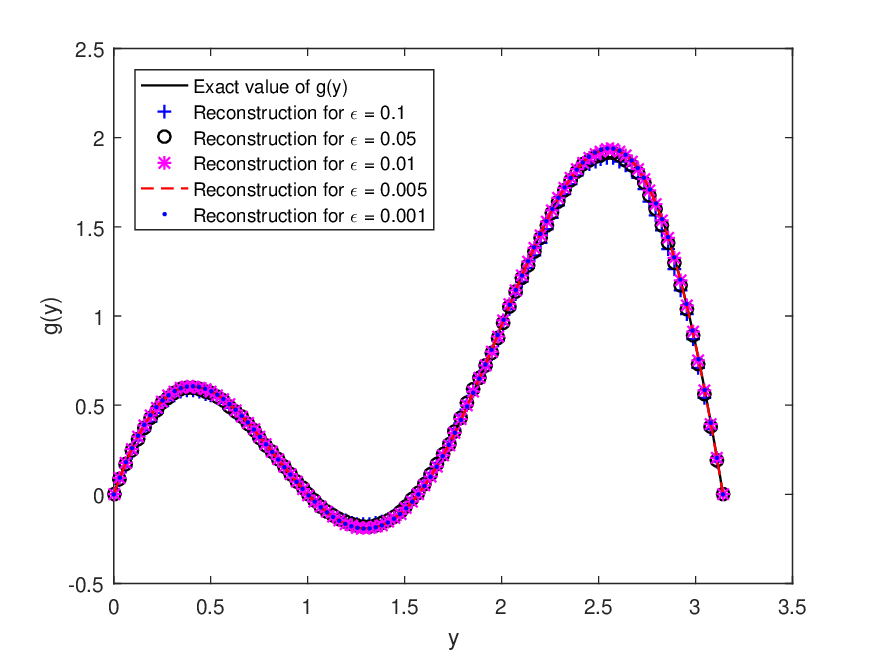}\label{fig01:sub3}}
\subfloat[ $\alpha=0.9,\beta=0.6,r=0.2$]{\includegraphics[width=0.35\textwidth]{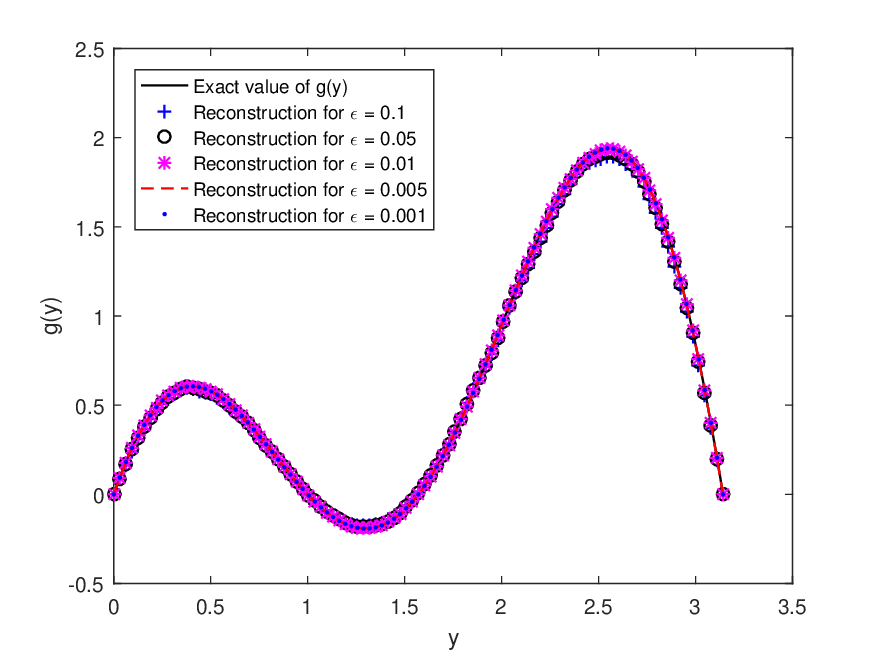}\label{fig01:sub4}}
\caption{Numerical inversions by the exponential-type Tikhonov regularization method for Example 1.}\label{fig01}
\end{figure}

\begin{figure}[H] 
	\centering
	\subfloat[ $\alpha=0.7,\beta=0.3,r=-1$]{\includegraphics[width=0.35\textwidth]{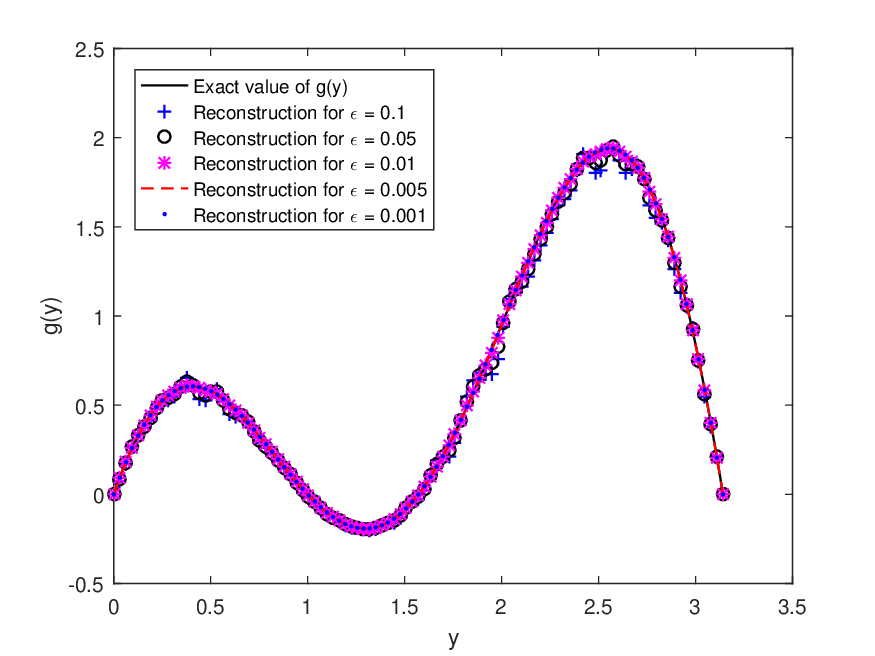}\label{fig011:sub1}}
	\subfloat[ $\alpha=0.9,\beta=0.6,r=-1$]{\includegraphics[width=0.35\textwidth]{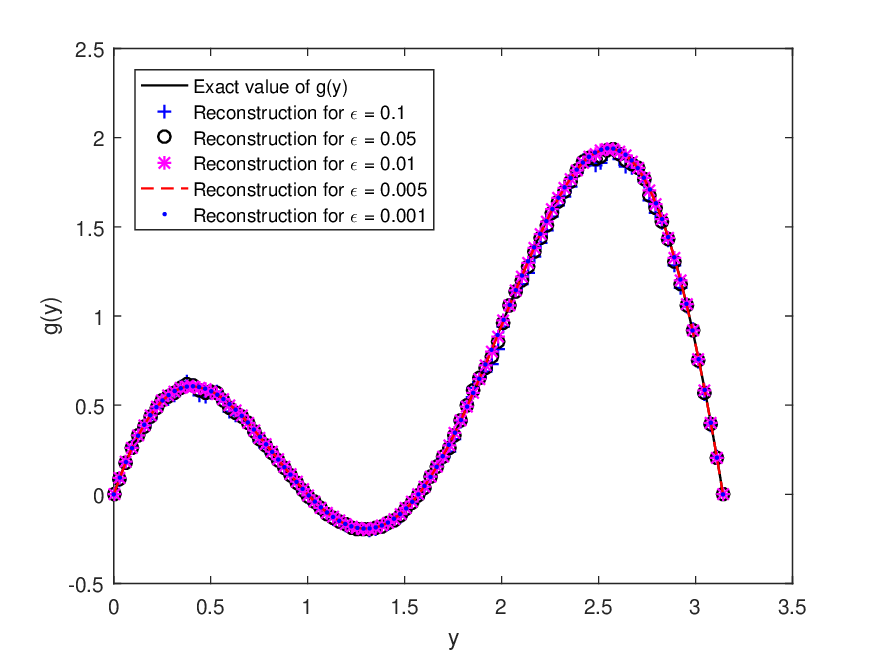}\label{fig011:sub2}}
	\\
	\subfloat[ $\alpha=0.7,\beta=0.3,r=0.2$]{\includegraphics[width=0.35\textwidth]{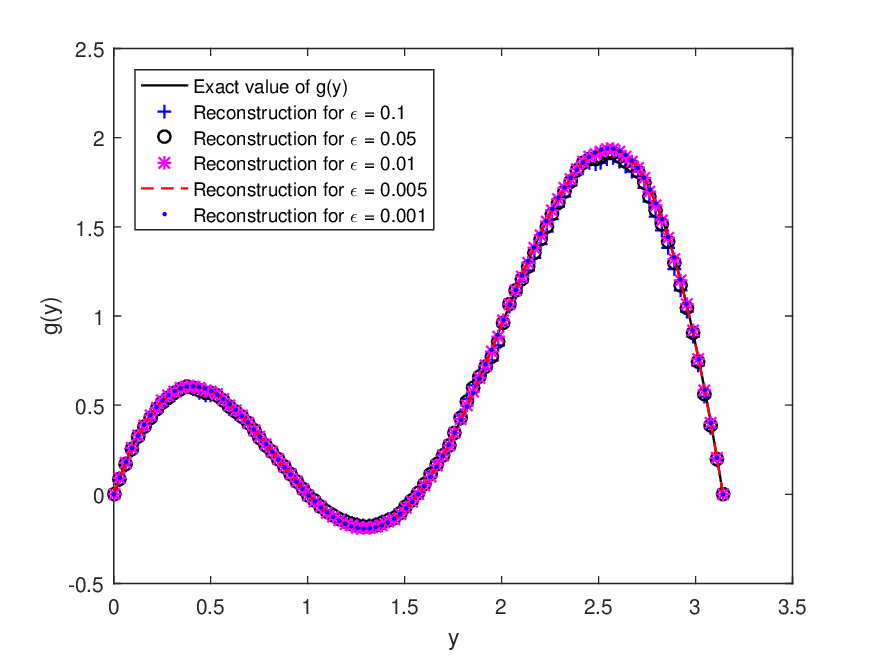}\label{fig011:sub3}}
	\subfloat[ $\alpha=0.9,\beta=0.6,r=0.2$]{\includegraphics[width=0.35\textwidth]{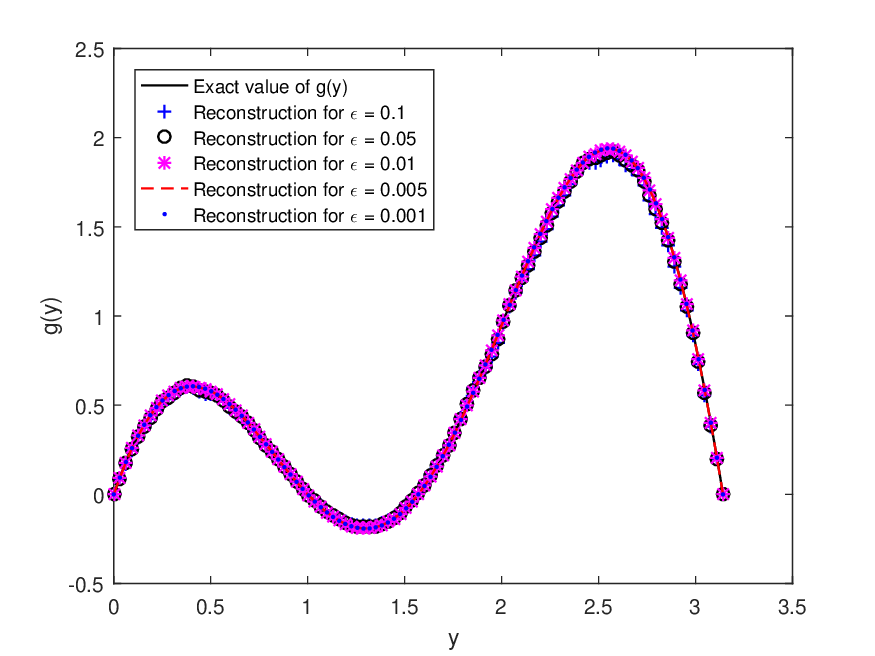}\label{fig011:sub4}}
	\caption{Numerical inversions by the exponential-type quasi-boundary regularization for Example 1.}\label{fig011}
\end{figure}

{\small
\begin{longtable}{ccccccccc}
\caption{Comparison of relative errors for different regularization methods in Example 1}
\label{tab01} \\
\hline
\multirow{2}{*}{$\alpha$} 
& \multirow{2}{*}{$\beta$} 
& \multirow{2}{*}{$\varepsilon$} 
& \multicolumn{2}{c}{Exponential Tikhonov} 
& \multicolumn{2}{c}{Exponential quasi-bound} 
& \multicolumn{2}{c}{General Tikhonov} \\
\cline{4-9}
& & 
& $r=-1$ 
& $r=0.2$ 
& $r=-1$ 
& $r=0.2$ 
& $r=0$ 
& $r=1$ \\
\hline
\endfirsthead

\hline
\multirow{2}{*}{$\alpha$} 
& \multirow{2}{*}{$\beta$} 
& \multirow{2}{*}{$\varepsilon$} 
& \multicolumn{2}{c}{Exponential Tikhonov} 
& \multicolumn{2}{c}{Exponential quasi-bound} 
& \multicolumn{2}{c}{General Tikhonov} \\
\cline{4-9}
& & 
& $r=-1$ 
& $r=0.2$ 
& $r=-1$ 
& $r=0.2$ 
& CTR 
& GTR: $r=1$ \\
\hline
\endhead

\hline
\endfoot

\hline
\endlastfoot

\multirow{5}{*}{0.9} 
& \multirow{5}{*}{0.6} 
& 0.1   
& 3.0928e-02 & 2.8189e-02 
& 3.1771e-02 & 2.9035e-02 
& 3.1548e-02 & 2.9119e-02 \\

& & 0.05   
& 1.5850e-02 & 1.4489e-02 
& 1.6021e-02 & 1.5143e-02 
& 1.6295e-02 & 1.5050e-02 \\

& & 0.01  
& 3.2452e-03 & 3.1685e-03 
& 3.2291e-03 & 3.2987e-03 
& 3.3705e-03 & 3.3445e-03 \\

& & 0.005  
& 1.6326e-03 & 1.6463e-03 
& 1.6211e-03 & 1.6871e-03 
& 1.6978e-03 & 1.7844e-03 \\

& & 0.001  
& 3.6035e-04 & 3.7585e-04 
& 3.5780e-04 & 3.7505e-04 
& 3.7241e-04 & 4.4294e-04 \\
\end{longtable}
}

\textbf{Example 2.}
Let the exact boundary value be
\[
f(y)=\frac{\sqrt{3}}{3\pi}\bigl(y^3-1\bigr)\sin(3y),
\]
and let the exact source term be the following piecewise smooth function:
\[
g(y)=
\begin{cases}
2y, & 0\le y\le \frac{\pi}{4},\\
1, & \frac{\pi}{4}<y\le \frac{\pi}{2},\\
\sin\!\left(5y+\frac{\pi}{2}\right)+1, & \frac{\pi}{2}<y\le \pi.
\end{cases}
\]

Numerical simulations are carried out for two sets of fractional parameters, 
$(\alpha,\beta)=(0.7,0.3)$ and $(\alpha,\beta)=(0.9,0.6)$. 
Figure \ref{fig02}--\ref{fig02} displays the exact source term and the reconstructed results obtained by the two proposed exponential-type regularization methods. 
The relative errors of the exponential-type Tikhonov regularization method, the exponential quasi-boundary value regularization method, the classical Tikhonov regularization method, and the generalized Tikhonov regularization method are compared in Table \ref{tab02}. 
Since the numerical results for different pairs of $(\alpha,\beta)$ show similar behavior, Table \ref{tab01} reports only the relative errors for $(\alpha,\beta)=(0.9,0.6)$.

\begin{figure}[H] 
\centering
\subfloat[ $\alpha=0.7,\beta=0.3,r=-1$]{\includegraphics[width=0.35\textwidth]{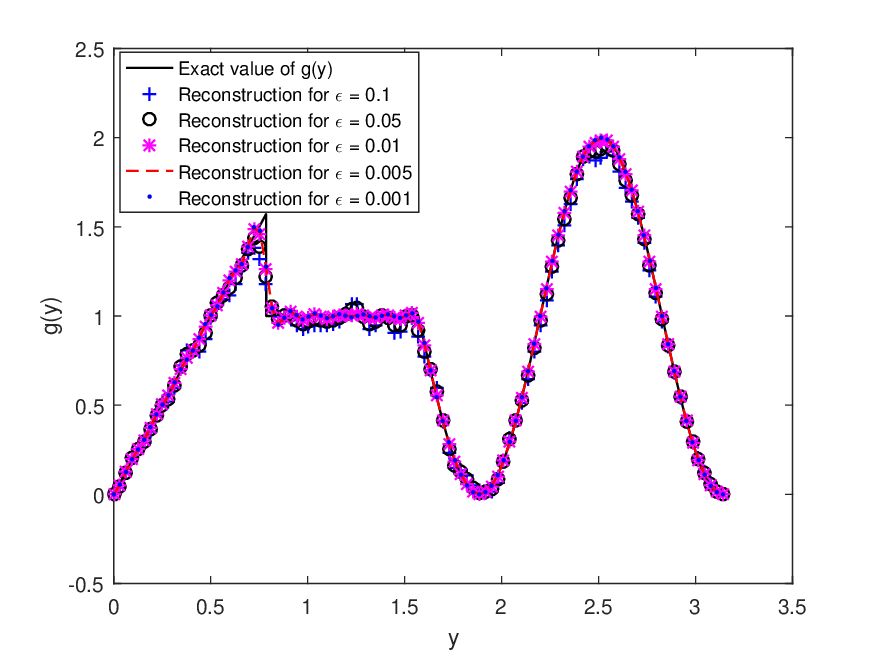}\label{fig02:sub1}}
\subfloat[ $\alpha=0.9,\beta=0.6,r=-1$]{\includegraphics[width=0.35\textwidth]{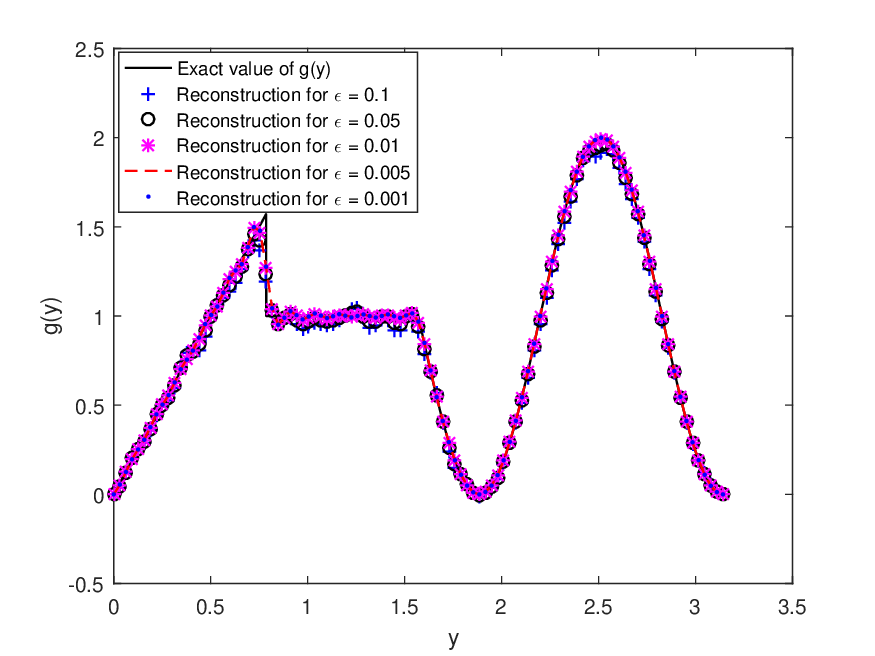}\label{fig02:sub2}}
\\
\subfloat[ $\alpha=0.7,\beta=0.3,r=0.2$]{\includegraphics[width=0.35\textwidth]{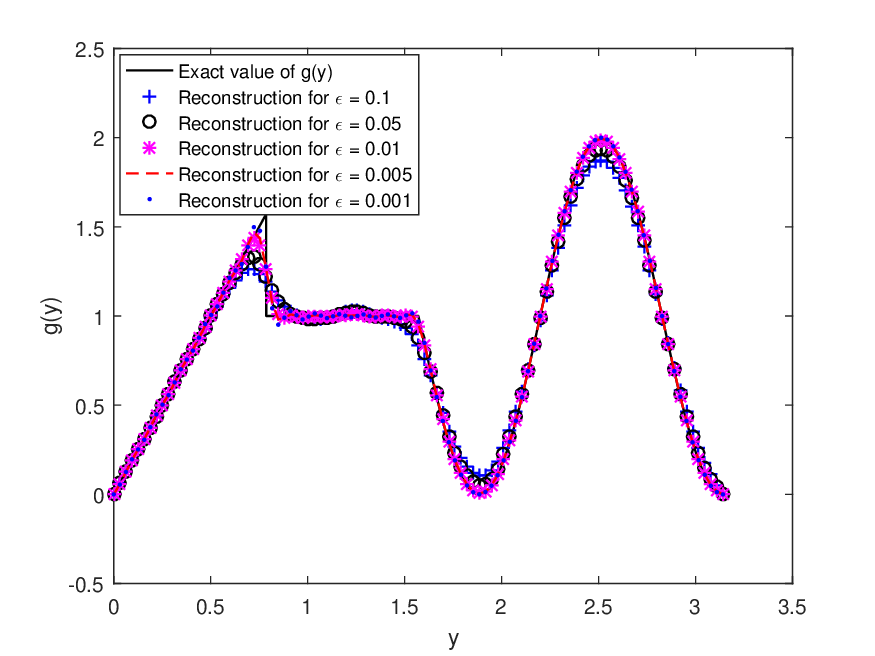}\label{fig02:sub3}}
\subfloat[ $\alpha=0.9,\beta=0.6,r=0.2$]{\includegraphics[width=0.35\textwidth]{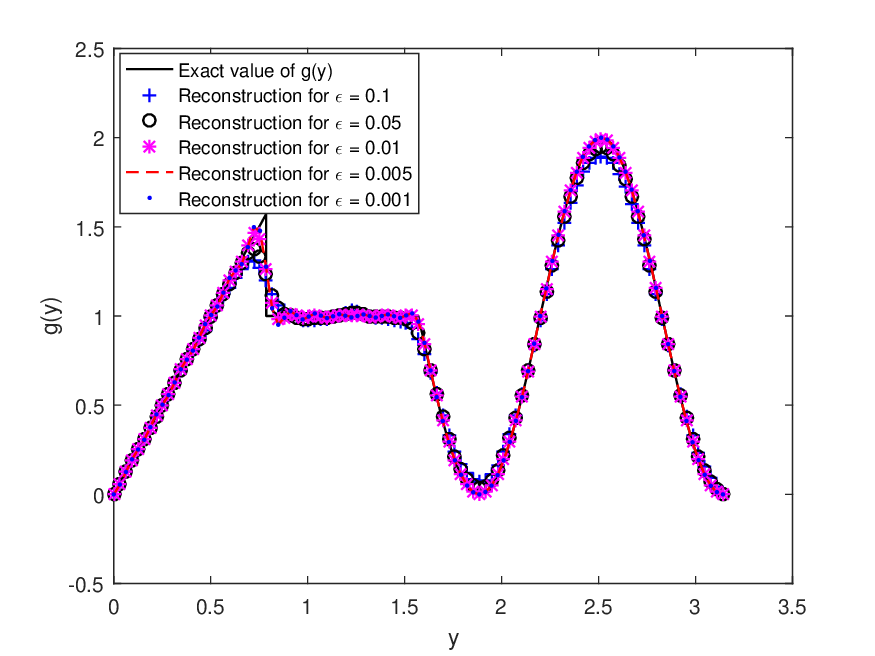}\label{fig02:sub4}}
\caption{Numerical inversions by the exponential-type Tikhonov regularization method for Example 2.}\label{fig02}
\end{figure}

\begin{figure}[H] 
	\centering
	\subfloat[ $\alpha=0.7,\beta=0.3,r=-1$]{\includegraphics[width=0.35\textwidth]{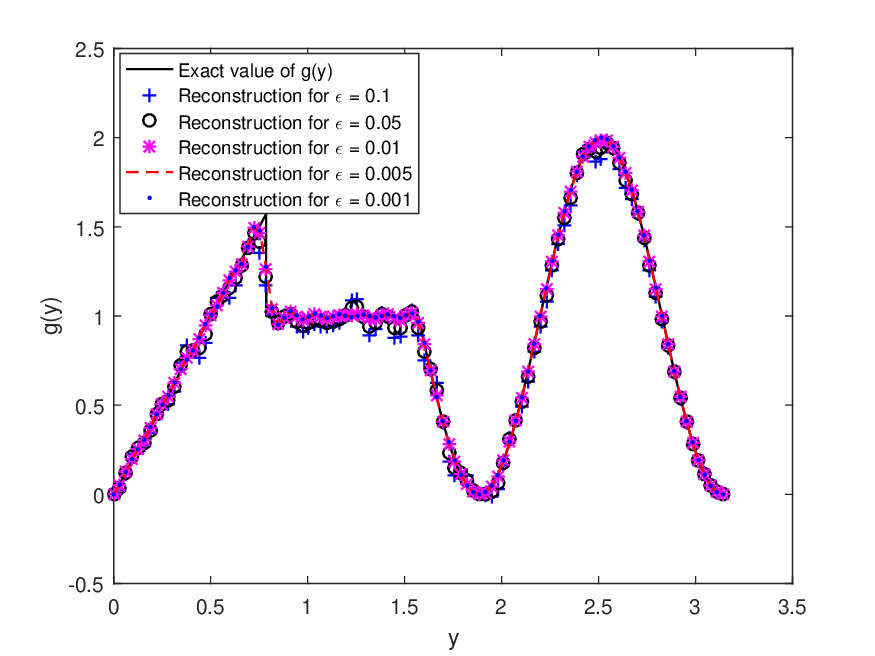}\label{fig021:sub1}}
	\subfloat[ $\alpha=0.9,\beta=0.6,r=-1$]{\includegraphics[width=0.35\textwidth]{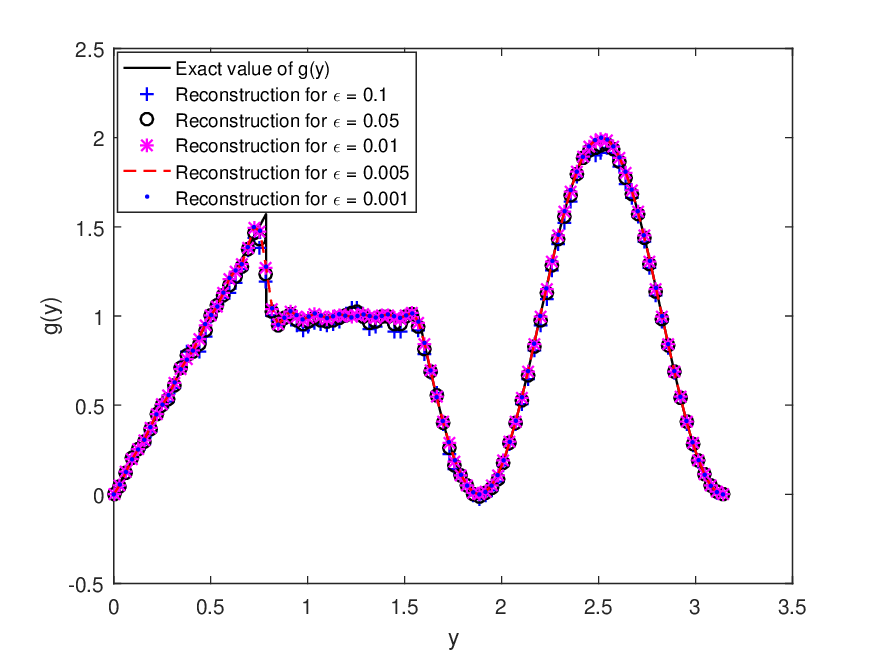}\label{fig021:sub2}}
	\\
	\subfloat[ $\alpha=0.7,\beta=0.3,r=0.2$]{\includegraphics[width=0.35\textwidth]{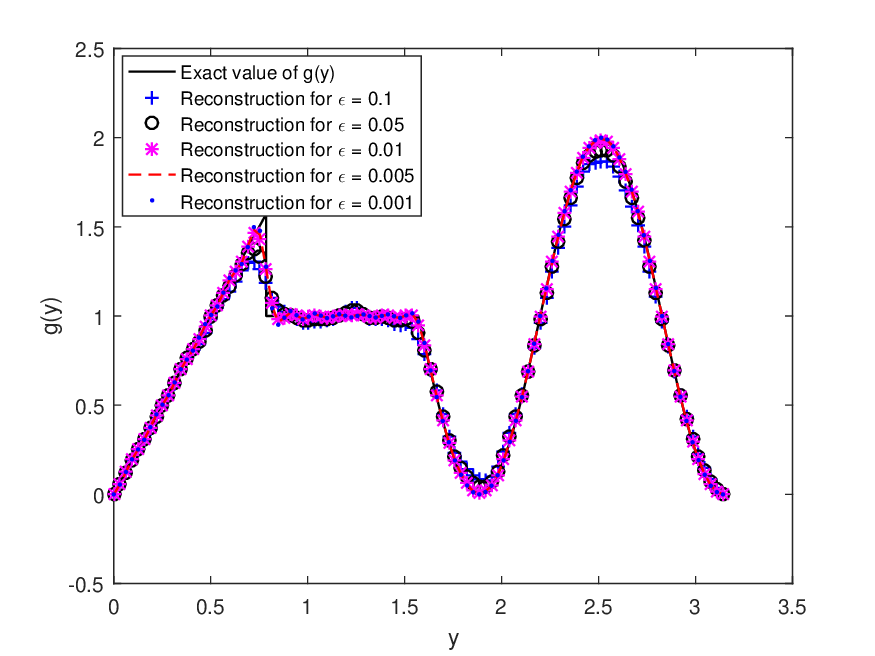}\label{fig021:sub3}}
	\subfloat[ $\alpha=0.9,\beta=0.6,r=0.2$]{\includegraphics[width=0.35\textwidth]{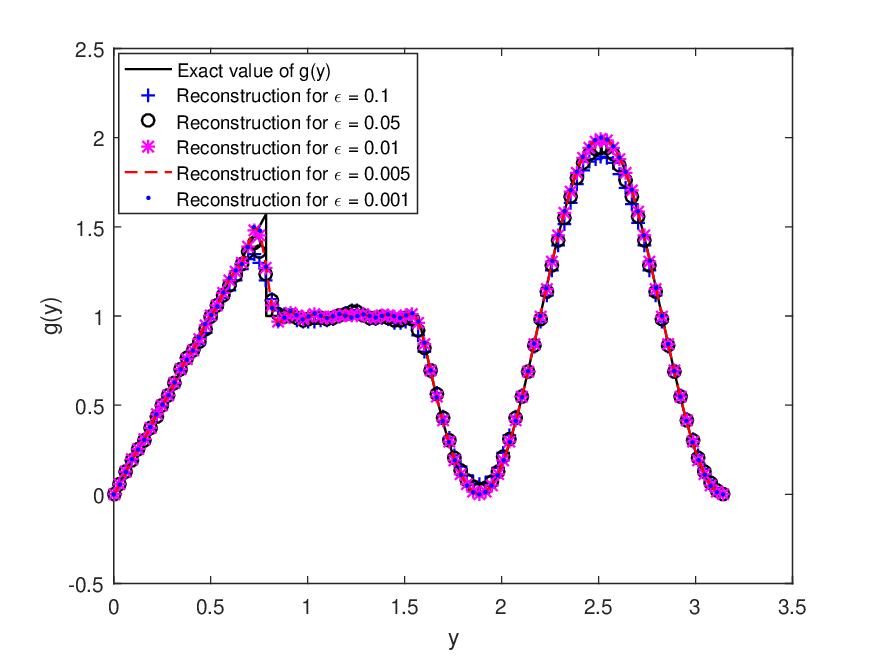}\label{fig021:sub4}}
	\caption{Numerical inversions by the exponential-type quasi-boundary regularization for Example 1.}\label{fig021}
\end{figure}

{\small
\begin{longtable}{ccccccccc}
\caption{Comparison of relative errors for different regularization methods in Example 2}
\label{tab02} \\
\hline
\multirow{2}{*}{$\alpha$} 
& \multirow{2}{*}{$\beta$} 
& \multirow{2}{*}{$\varepsilon$} 
& \multicolumn{2}{c}{Exponential Tikhonov} 
& \multicolumn{2}{c}{Exponential quasi-bound} 
& \multicolumn{2}{c}{General Tikhonov} \\
\cline{4-9}
& & 
& $r=-1$ 
& $r=0.2$ 
& $r=-1$ 
& $r=0.2$ 
& $r=0$ 
& $r=1$ \\
\hline
\endfirsthead

\hline
\multirow{2}{*}{$\alpha$} 
& \multirow{2}{*}{$\beta$} 
& \multirow{2}{*}{$\varepsilon$} 
& \multicolumn{2}{c}{Exponential Tikhonov} 
& \multicolumn{2}{c}{Exponential quasi-bound} 
& \multicolumn{2}{c}{General Tikhonov} \\
\cline{4-9}
& & 
& $r=-1$ 
& $r=0.2$ 
& $r=-1$ 
& $r=0.2$ 
& CTR 
& GTR: $r=1$ \\
\hline
\endhead

\hline
\endfoot

\hline
\endlastfoot

\multirow{5}{*}{0.9} 
& \multirow{5}{*}{0.6} 
& 0.1   
& 4.7986e-02 & 5.9185e-02 
& 4.8259e-02 & 5.4976e-02 
& 5.1300e-02 & 6.8002e-02 \\

& & 0.05   
& 3.2610e-02 & 4.0168e-02 
& 3.2461e-02 & 3.6836e-02 
& 3.4130e-02 & 4.6358e-02 \\

& & 0.01  
& 2.5230e-02 & 2.6313e-02 
& 2.5214e-02 & 2.5618e-02  
& 2.5304e-02 & 2.7302e-02 \\

& & 0.005  
& 2.4965e-02 & 2.4965e-02 
& 2.4965e-02 & 2.4965e-02  
& 2.4965e-02 & 2.4965e-02 \\

& & 0.001  
& 2.4930e-02 & 2.4930e-02 
& 2.4930e-02 & 2.4930e-02  
& 2.4930e-02 & 2.4930e-02 \\
\end{longtable}
}

The results in Figures \ref{fig01}--\ref{fig021} and Tables \ref{tab01}--\ref{tab02} show that the proposed exponential-type regularization methods are stable with respect to data noise and provide reliable reconstructions of the source term. For the smooth source in Example~1, the relative errors decrease as the noise level becomes smaller, and the exponential-type methods generally achieve better accuracy than the general Tikhonov method, where the case $r=0$ corresponds to the classical Tikhonov regularization and $r=1$ to the generalized Tikhonov regularization. In particular, the exponential Tikhonov method with $r=0.2$ gives the smallest errors for relatively large noise levels, while the exponential quasi-boundary value method with $r=-1$ becomes slightly more accurate for smaller noise levels. For the piecewise smooth source in Example~2, the relative errors first decrease as the noise level is reduced and then approach a stable error plateau, mainly due to the limited smoothness of the source function and the spectral truncation effect. In this case, the choice $r=-1$ is more suitable, and the exponential quasi-boundary value method is slightly more accurate than the exponential Tikhonov method for several noise levels. Overall, the proposed exponential-type Tikhonov and exponential quasi-boundary value regularization methods are competitive with, and in many cases more accurate than, the classical and generalized Tikhonov regularization methods.

%==============================================
\section{Conclusion} \label{sec:Conclusion}
%==============================================

An inverse source problem for a fractional-order elliptic equation with singular coefficients was studied in this paper. For the corresponding direct problem, a formal solution was derived and the well-posedness of the solution was established. For the inverse problem, the ill-posedness was analyzed and a H\"older-type conditional stability estimate was obtained in a Hilbert scale associated with exponential operators.

To reconstruct the unknown source term, two regularization methods were developed. First, an exponential-type Tikhonov regularization method was introduced, and convergence estimates were derived for both a priori and a posteriori choices of the regularization parameter. Motivated by the quasi-boundary value idea, an exponential-type quasi-boundary value regularization method was then proposed. This method incorporates the exponential regularization term directly into the modified boundary condition and provides an alternative way to stabilize the inverse source problem. Corresponding convergence estimates were also established. In addition, finite-dimensional spectral approximation results showed that both proposed methods can be applied to general square-integrable source terms, without requiring the exact source to belong to an exponential-type Hilbert scale; see Theorems \ref{the09} and \ref{the10}.

The numerical experiments confirmed the stability and effectiveness of the proposed methods for both smooth and piecewise smooth sources. The exponential-type methods generally achieved smaller relative errors than the classical Tikhonov method ($r=0$) and the generalized Tikhonov method ($r=1$). For the smooth source, positive values of $r$ were often more effective, whereas for the piecewise smooth source, the choice $r=-1$ gave better results. These results demonstrate the practical applicability of the proposed methods for noisy data.

Nevertheless, the implementation of the proposed methods relies on the eigen-decomposition of the elliptic operator and the numerical evaluation of the three-parameter Mittag-Leffler function $E_{\alpha,m,n}(z)$. These requirements may lead to computational challenges for large-scale, high-dimensional, or variable-coefficient problems. Future work will focus on developing more efficient numerical algorithms and exploring alternative regularization techniques for fractional inverse source problems with more complex structures.

\section*{Acknowledgments}
This work is supported by National Natural Science Foundation of China (12261004, 12171248), Guangdong Basic and Applied Basic Research Foundation(2025A1515012248), Innovation Team Project of Regular Universities in Guangdong Province (2025KCXTD037), Tertiary Education Scientific research project of Guangzhou Municipal Education Bureau (2024312515).

\medskip

\noindent
\textbf{Conflict of interest:} The authors declare that there is no conflict of interest regarding this
submitted manuscript.

\section*{Appendix}	

\textbf{Proof of Theorem \ref{the07}.}

Direct calculation yields
\begin{align*}
\left\|g_{\tau,r}^\delta(y)\!-\!g_{\tau,r}(y)\right\|&=\left\|\sum_{k=1}^\infty \frac{\frac{\eta_{2}}{\sqrt{\lambda_k}}l^{\alpha+\beta}E_{\alpha,1+\frac{\beta}{\alpha}, 1+\frac{2\beta}{\alpha}}\left(-\sqrt{\lambda_k}l^{\alpha+\beta}\right)} {\left(\frac{\eta_{2}}{\sqrt{\lambda_k}}l^{\alpha+\beta}E_{\alpha,1+\frac{\beta}{\alpha}, 1+\frac{2\beta}{\alpha}}\left(-\sqrt{\lambda_k}l^{\alpha+\beta}\right)\right)^2\!+\! \tau\exp(\lambda_k^r)}\left(\hat{h}_k^\delta\!-\!\hat{h}_k\right)\,\varphi_k(y)\right\|\\
&\leq\delta \sup_k\frac{\frac{\eta_{2}}{\sqrt{\lambda_k}}l^{\alpha+\beta} E_{\alpha,1+\frac{\beta}{\alpha},1+\frac{2\beta}{\alpha}}\left(-\sqrt{\lambda_k} l^{\alpha+\beta}\right)}{\left(\frac{\eta_{2}}{\sqrt{\lambda_k}}l^{\alpha+\beta} E_{\alpha,1+\frac{\beta}{\alpha},1+\frac{2\beta}{\alpha}}\left(-\sqrt{\lambda_k}l^{\alpha+\beta} \right)\right)^2+\tau\exp(\lambda_k^r)}.
\end{align*}
and 
\begin{align*}
\left\|g_{\tau,r}(y)\!-\!g(y)\right\|^2& =\left\|\sum_{k=1}^{\infty}\frac{\frac{\eta_{2}}{\sqrt{\lambda_k}}l^{\alpha\!+\!\beta} E_{\alpha,1\!+\!\frac{\beta}{\alpha},1\!+\!\frac{2\beta}{\alpha}} \left(-\sqrt{\lambda_k}l^{\alpha\!+\!\beta}\right)}{\left(\frac{\eta_{2}}{\sqrt{\lambda_k}} l^{\alpha\!+\!\beta}E_{\alpha,1\!+\!\frac{\beta}{\alpha},1\!+\!\frac{2\beta}{\alpha}} \left(-\sqrt{\lambda_k}l^{\alpha\!+\!\beta}\right)\right)^2\!+\!\tau\exp(\lambda_k^r)}\hat{h}_k\, \varphi_k(y)\!-\!\sum_{k=1}^{\infty}g_k\varphi_k(y)\right\|^2 \\
&=\left\|\sum_{k=1}^\infty\frac{\left(\frac{\eta_{2}}{\sqrt{\lambda_k}}l^{\alpha\!+\!\beta} E_{\alpha,1\!+\!\frac{\beta}{\alpha},1\!+\!\frac{2\beta}{\alpha}}\left(-\sqrt{\lambda_k} l^{\alpha\!+\!\beta}\right)\right)^2}{\left(\frac{\eta_{2}}{\sqrt{\lambda_k}}l^{\alpha\!+\!\beta} E_{\alpha,1\!+\!\frac{\beta}{\alpha},1\!+\!\frac{2\beta}{\alpha}}\left(-\sqrt{\lambda_k} l^{\alpha\!+\!\beta}\right)\right)^2\!+\!\tau\exp(\lambda_k^r)}g_k\varphi_k(y)\!-\!\sum_{k=1}^\infty g_k\varphi_k(y)\right\|^2 \\
%&=\left\|\sum_{k=1}^\infty\frac{\tau\exp(\lambda_k^r)}{\left(\frac{\eta_{2}}{\sqrt{\lambda_k}} l^{\alpha+\beta}E_{\alpha,1+\frac{\beta}{\alpha},1+\frac{2\beta}{\alpha}} \left(-\sqrt{\lambda_k}l^{\alpha+\beta}\right)\right)^2+\tau\exp(\lambda_k^r)}g_k\varphi_k(y)\right\|\\
&=\sum_{k=1}^\infty\left(\frac{\tau\exp(\lambda_k^r)}{\left(\frac{\eta_{2}}{\sqrt{\lambda_k}} l^{\alpha+\beta}E_{\alpha,1+\frac{\beta}{\alpha},1+\frac{2\beta}{\alpha}} \left(-\sqrt{\lambda_k}l^{\alpha+\beta}\right)\right)^2+\tau\exp(\lambda_k^r)}\right)^2\left|g_k\right|^2. 
\end{align*}

\textbf{Case (i): $r\le0$.} 
When $r\le0$, since $\lim\limits_{k\to\infty}\lambda_k=+\infty$, there exists a constant $C>0$ such that 
$1\le \exp(\lambda_k^r)\le C$ for all $k$. 
Using estimate \eqref{equ208} and absorbing all constants into $C$, we obtain
\begin{equation}\label{equ508_new}
\|g_{\tau,r}^{\delta}-g_{\tau,r}\|
\leq 
C\,\frac{\delta}{\sqrt{\tau}}.
\end{equation}

For $p\ge2$, we have
\begin{align*}
\|g_{\tau,r}-g\|^2
&=\sum_{k=1}^\infty
\left(
\frac{\tau\exp(\lambda_k^r)}
{\left(
\frac{\eta_{2}}{\sqrt{\lambda_k}}l^{\alpha+\beta}
E_{\alpha,1+\frac{\beta}{\alpha},1+\frac{2\beta}{\alpha}}
\!\left(-\sqrt{\lambda_k}l^{\alpha+\beta}\right)
\right)^2
+\tau\exp(\lambda_k^r)}
\right)^{\!2}\! g_k^2 \\
&\le 
\sum_{k=1}^\infty
\left(
\frac{C\,\tau\,\lambda_k^{2-p}}{1+\tau\lambda_k^2}
\right)^{\!2}\!
\lambda_k^{2p}g_k^2
\le 
C^2\,\tau^2\,\|g\|_{D((-L)^p)}^2.
\end{align*}
Hence,
\[
\|g_{\tau,r}-g\|\le C\,\tau\,\|g\|_{D((-L)^p)}.
\]
Combining this with \eqref{equ508_new} and choosing $\tau=\delta^{2/3}$ yields
\[
\|g_{\tau,r}^{\delta}-g\|
\le
C\left(1+\|g\|_{D((-L)^p)}\right)\delta^{2/3}.
\]

For $0<p<2$, note that
\[
\frac{C\,\tau\,\lambda_k^{2-p}}{1+\tau\lambda_k^2}
\le 
C\,\sup_{t>0}\frac{t^{2-p}}{1+\tau t^2}\,\tau
\le 
C\,\tau^{p/2}.
\]
Therefore,
\[
\|g_{\tau,r}-g\|\le C\,\tau^{p/2}\|g\|_{D((-L)^p)}.
\]
Combining with \eqref{equ508_new} and taking $\tau=\delta^{2/(p+1)}$ gives
\[
\|g_{\tau,r}^{\delta}-g\|
\le
C\left(1+\|g\|_{D((-L)^p)}\right)\delta^{p/(p+1)}.
\]

\textbf{Case (ii): $r>0$.} 
Using estimate \eqref{equ208} and absorbing all constants into a generic constant $C>0$, we obtain
\begin{align*}
\sup_{k\ge1}\frac{\frac{\eta_{2}}{\sqrt{\lambda_k}}l^{\alpha+\beta}
E_{\alpha,1+\frac{\beta}{\alpha},1+\frac{2\beta}{\alpha}}
\!\left(-\sqrt{\lambda_k}l^{\alpha+\beta}\right)}
{\left(\frac{\eta_{2}}{\sqrt{\lambda_k}}l^{\alpha+\beta}
E_{\alpha,1+\frac{\beta}{\alpha},1+\frac{2\beta}{\alpha}}
\!\left(-\sqrt{\lambda_k}l^{\alpha+\beta}\right)\right)^2
+\tau\exp(\lambda_k^r)}
&\le 
C\,\sup_{k\ge1}\frac{\lambda_k}{1+\tau\lambda_k^{2+r}}
\le C\,\tau^{-1/(2+r)}.
\end{align*}
Hence,
\begin{equation}\label{equ507_new}
\|g_{\tau,r}^{\delta}-g_{\tau,r}\|
\le 
C\,\frac{\delta}{\tau^{1/(2+r)}}.
\end{equation}

Since $r>0$ and $g\in D\!\left(\exp\!\left(\frac{(-L)^r}{2}\right)\right)$,  
it follows from the spectral representation and estimate \eqref{equ208} that
\begin{align*}
\|g_{\tau,r}-g\|^2
&=\sum_{k=1}^\infty
\left(
\frac{\tau\exp(\lambda_k^r)}
{\left(
\frac{\eta_{2}}{\sqrt{\lambda_k}}l^{\alpha+\beta}
E_{\alpha,1+\frac{\beta}{\alpha},1+\frac{2\beta}{\alpha}}
\!\left(-\sqrt{\lambda_k}l^{\alpha+\beta}\right)
\right)^2
+\tau\exp(\lambda_k^r)}
\right)^{\!2}\! |g_k|^2 \\[3pt]
&\le 
\sum_{k=1}^\infty
\left(
\frac{C\,\tau\,\lambda_k^{2}\exp(\lambda_k^r)}
{1+\tau\lambda_k^{2}\exp(\lambda_k^r)}
\right)^{\!2}\! |g_k|^2.
\end{align*}

Let $p>1$ be an arbitrary positive number. 
Noting that $\exp((p-1)\lambda_k^r/(p+1))$ grows faster than any power $\lambda_k^{4/(p+1)}$,  
we can bound the above by
\begin{align*}
\|g_{\tau,r}-g\|^2
&\le
C^2
\!\sup_{k\ge1}\!\left(
\frac{\tau\bigl(\lambda_k^2\exp(\lambda_k^r)\bigr)^{\frac{p}{p+1}}}
{1+\tau\lambda_k^2\exp(\lambda_k^r)}
\right)^{\!2}
\sum_{k=1}^\infty 
\exp(\lambda_k^r)\,|g_k|^2.
\end{align*}
A straightforward calculus argument yields
\[
\sup_{t>0}\frac{\tau\,t^{\frac{p}{p+1}}}{1+\tau t}
\le
C\,\tau^{\frac{1}{p+1}}.
\]
Thus,
\[
\|g_{\tau,r}-g\|
\le
C\,\tau^{\frac{1}{p+1}}\,\|g\|_{r,\mathrm{exp}}.
\]

Combining the two estimates and balancing the errors in 
\eqref{equ507_new}, we choose
\(
\tau=\delta^{\frac{(p+1)(2+r)}{(2+r)+(p+1)}},
\)
which minimizes the total error, yielding
\[
\|g_{\tau,r}^{\delta}-g\|
\le
C\!\left(1+\|g\|_{r,\mathrm{exp}}\right)
\delta^{\frac{(2+r)}{(2+r)+(p+1)}}.
\]
\hfill$\Box$

\medskip

\textbf{Proof of Theorem \ref{the08}.}

Let
\[
\sigma_k=
\frac{\eta_2}{\sqrt{\lambda_k}}l^{\alpha+\beta}
E_{\alpha,1+\frac{\beta}{\alpha},1+\frac{2\beta}{\alpha}}
\!\left(-\sqrt{\lambda_k}l^{\alpha+\beta}\right).
\]

\textbf{Case (i): $r\le0$.}
By the Morozov discrepancy principle,
\[
\rho\delta
=
\left\|
\sum_{k=1}^{\infty}
\frac{\tau e^{\lambda_k^r}}
{\sigma_k^2+\tau e^{\lambda_k^r}}
\hat h_k^\delta\varphi_k
\right\|.
\]
Since
\[
0<
\frac{\tau e^{\lambda_k^r}}
{\sigma_k^2+\tau e^{\lambda_k^r}}
<1
\]
and $\|\hat h^\delta-\hat h\|\le\delta$, we obtain
\[
(\rho-1)\delta
\le
\left\|
\sum_{k=1}^{\infty}
\frac{\tau e^{\lambda_k^r}}
{\sigma_k^2+\tau e^{\lambda_k^r}}
\hat h_k\varphi_k
\right\|.
\]
Using $\hat h_k=\sigma_k g_k$, it follows that
\[
(\rho-1)^2\delta^2
\le
\sum_{k=1}^{\infty}
\left(
\frac{\tau e^{\lambda_k^r}\sigma_k}
{\sigma_k^2+\tau e^{\lambda_k^r}}
\right)^2
|g_k|^2.
\]

For $r\le0$, $e^{\lambda_k^r}$ is uniformly bounded from above and below. Moreover, by estimate \eqref{equ208}, there exist constants $c,C>0$ such that
\[
\frac{c}{\lambda_k}\le \sigma_k\le \frac{C}{\lambda_k}.
\]
Hence,
\[
\frac{\tau e^{\lambda_k^r}\sigma_k\lambda_k^{-p}}
{\sigma_k^2+\tau e^{\lambda_k^r}}
\le
C\frac{\tau\lambda_k^{1-p}}
{1+\tau\lambda_k^2}.
\]
Therefore,
\[
(\rho-1)^2\delta^2
\le
C
\sum_{k=1}^{\infty}
\left(
\frac{\tau\lambda_k^{1-p}}
{1+\tau\lambda_k^2}
\right)^2
\lambda_k^{2p}|g_k|^2.
\]

If $p\ge1$, then
\[
\sup_{t\ge\lambda_1}
\frac{\tau t^{1-p}}{1+\tau t^2}
\le C\tau.
\]
If $0<p<1$, then a direct maximum estimate gives
\[
\sup_{t>0}
\frac{\tau t^{1-p}}{1+\tau t^2}
\le C\tau^{\frac{1+p}{2}}.
\]
Consequently,
\[
(\rho-1)^2\delta^2
\le
\begin{cases}
C C_5^2\tau^2, & p\ge1,\\
C C_5^2\tau^{p+1}, & 0<p<1.
\end{cases}
\]
Thus,
\[
\frac1{\sqrt{\tau}}
\le
\begin{cases}
C\left(\frac{C_5}{\rho-1}\right)^{\frac12}\delta^{-\frac12},
& p\ge1,\\
C\left(\frac{C_5}{\rho-1}\right)^{\frac1{p+1}}\delta^{-\frac1{p+1}},
& 0<p<1.
\end{cases}
\]

Using the perturbation estimate
\[
\|g_{\tau,r}^{\delta}-g_{\tau,r}\|
\le
C\frac{\delta}{\sqrt{\tau}},
\]
we obtain
\[
\|g_{\tau,r}^{\delta}-g_{\tau,r}\|
\le
\begin{cases}
C\left(\frac{C_5}{\rho-1}\right)^{\frac12}\delta^{\frac12},
& p\ge1,\\
C\left(\frac{C_5}{\rho-1}\right)^{\frac1{p+1}}\delta^{\frac{p}{p+1}},
& 0<p<1.
\end{cases}
\]

Next,
\[
\|g_{\tau,r}-g\|_{D((-L)^p)}
\le
\|g\|_{D((-L)^p)}
\le C_5.
\]
Moreover,
\[
K(g_{\tau,r}-g)
=
-\sum_{k=1}^{\infty}
\frac{\tau e^{\lambda_k^r}}
{\sigma_k^2+\tau e^{\lambda_k^r}}
\hat h_k\varphi_k .
\]
By adding and subtracting $\hat h_k^\delta$ and using the discrepancy principle, we get
\[
\|K(g_{\tau,r}-g)\|
\le
\rho\delta+\delta
=
(1+\rho)\delta.
\]
Applying the conditional stability estimate \eqref{equ411} with the index $2p$ in place of $p$ yields
\[
\|g_{\tau,r}-g\|
\le
C
C_5^{\frac1{p+1}}
(1+\rho)^{\frac{p}{p+1}}
\delta^{\frac{p}{p+1}}.
\]
Combining the last two estimates gives the assertion in Case (i).

\textbf{Case (ii): $r>0$.}
Since $g_{\tau,r}^{\delta}$ minimizes $J_{\tau,r}$, we have
\[
\|K g_{\tau,r}^{\delta}-\hat h^\delta\|^2
+\tau\|g_{\tau,r}^{\delta}\|_{r,\mathrm{exp}}^2
\le
\|K g-\hat h^\delta\|^2
+\tau\|g\|_{r,\mathrm{exp}}^2 .
\]
Using $\|K g-\hat h^\delta\|=\|\hat h-\hat h^\delta\|\le\delta$ and the discrepancy principle,
\[
\rho^2\delta^2+\tau\|g_{\tau,r}^{\delta}\|_{r,\mathrm{exp}}^2
\le
\delta^2+\tau\|g\|_{r,\mathrm{exp}}^2.
\]
Hence,
\[
\|g_{\tau,r}^{\delta}\|_{r,\mathrm{exp}}
\le
\|g\|_{r,\mathrm{exp}}
\le C_6.
\]
Since $r>0$, the embedding
\[
D\!\left(\exp\!\left(\frac{(-L)^r}{2}\right)\right)
\subset D((-L)^{p/2})
\]
holds for every $p>0$. Therefore,
\[
\|g_{\tau,r}^{\delta}-g\|_{D((-L)^{p/2})}
\le
2C_6.
\]
Furthermore,
\[
\|K(g_{\tau,r}^{\delta}-g)\|
\le
\|K g_{\tau,r}^{\delta}-\hat h^\delta\|
+
\|\hat h^\delta-\hat h\|
\le
(1+\rho)\delta.
\]
Applying the conditional stability estimate \eqref{equ411} yields
\[
\|g_{\tau,r}^{\delta}-g\|
\le
C\,
\|g_{\tau,r}^{\delta}-g\|_{D((-L)^{p/2})}^{\frac{2}{p+2}}\,
\|K(g_{\tau,r}^{\delta}-g)\|^{\frac{p}{p+2}}
\le
C\,(2C_6)^{\frac{2}{p+2}}\,
(1+\rho)^{\frac{p}{p+2}}\,
\delta^{\frac{p}{p+2}}.
\]
\hfill $\Box$.

%\section*{References}

\end{document}